\documentclass[12pt]{article}
\usepackage{fancyhdr}
\usepackage[centertags]{amsmath}
\usepackage{amsfonts}
\usepackage{graphics}
\usepackage{graphicx}
\usepackage{wrapfig}
\usepackage{xcolor}
\usepackage{amssymb}
\usepackage{amsthm}
\usepackage{newlfont}
\usepackage{url}
\usepackage{epsfig}
\usepackage{ifthen}
\usepackage{ifpdf}
\usepackage{rotating}
\usepackage{caption}
\usepackage{subcaption}
\usepackage{amsmath}
\usepackage{subfloat}
\input amssym.def
\input amssym
\input cyracc.def

\newtheorem{lemma}{Lemma}[section]
\newtheorem{prop}{Proposition}[section]
\newtheorem{te}{Theorem}[section]
\newtheorem{cor}{Corollary}[section]

\newtheorem{ass}{A}

\newcommand{\be}{\begin{equation}}
\newcommand{\ee}{\end{equation}}
\newcommand{\ba}{\begin{array}}
\newcommand{\ea}{\end{array}}
\newcommand{\bee}{\begin{eqnarray*}}
\newcommand{\eee}{\end{eqnarray*}}
\newcommand{\bea}{\begin{eqnarray}}
\newcommand{\eea}{\end{eqnarray}}

\newcommand{\II}{\mathbb{I}}
\newcommand{\WW}{\mathbb{W}}
\newcommand{\RR}{\mathbb{R}}
\newcommand{\BB}{\mathbb{B}}
\newcommand{\AAA}{\mathbb{A}}
\newcommand{\DD}{\mathbb{D}}
\newcommand{\MM}{\mathbb{M}}
\newcommand{\LL}{\mathbb{L}}
\newcommand{\GG}{\mathbb{G}}
\newcommand{\UU}{\mathbb{U}}
\newcommand{\VV}{\mathbb{V}}
\newcommand{\mx}{\mathbf{x}}
\newcommand{\mb}{\mathbf{b}}
\newcommand{\mc}{\mathbf{c}}

\newcommand{\pp}{\mathbf{p}}

\begin{document}

\title{EFIX: Exact Fixed Point Methods for Distributed Optimization}

\author{ Du\v{s}an Jakoveti\'c \footnotemark[1] \and Nata\v sa Kreji\'c \footnotemark[1]
 \and Nata\v sa Krklec Jerinki\'c \footnotemark[1]
 }

\maketitle

\begin{abstract} {We consider strongly convex distributed consensus optimization over
connected networks. EFIX, the proposed method, is derived using 
quadratic penalty approach. In more detail, we use the standard
reformulation -- transforming the original problem into a constrained
problem in a higher dimensional space -- to define a sequence of
suitable quadratic penalty subproblems with increasing penalty
parameters. For quadratic objectives, the corresponding sequence
consists of quadratic penalty subproblems. For the generic strongly
convex case, the objective function is approximated with a quadratic
model and hence the sequence of the resulting penalty subproblems is
again quadratic. EFIX is then derived by solving each of the quadratic
penalty subproblems via a fixed point (R)-linear solver, e.g., Jacobi Over-Relaxation method. The exact convergence is proved as well as the worst case complexity of order  $\cal{O}(\epsilon^{-1})$  for the quadratic case. In the case of  strongly
convex generic functions, the standard result for penalty methods is obtained. 
Numerical results indicate that the method is highly competitive with
state-of-the-art exact first order methods, requires smaller
computational and communication effort, and is robust to the choice of
algorithm parameters.
}

%\begin{abstract} {We consider strongly convex distributed optimization problems in  decentralized connected network. EFIX, the proposed method, fits into the framework of quadratic penalty methods. Using the standard reformulation of distributed optimization problems to constrained problems in larger space, a sequence of quadratic penalty problems with increasing penalty parameters is defined. For quadratic objective function the sequence consists of  quadratic penalty subproblems. For strongly convex generic functions the objective function is approximated with a quadratic model and hence the sequence of resulting penalty subproblems is again a sequence of quadratic problems. Each of the quadratic penalty problems is solved by a fixed point linear solver. The proposed method converges to the exact solution of the distributed problem in the case of quadratic objective function. We also prove the worst case complexity of order $\cal{O}(\epsilon^{-1})$. The convergence result for generic strongly convex function corresponds to the result in centralized optimization. Numerical results indicate that the method is highly competitive with the state-of-the-art first order exact methods, requires smaller computational and communication effort  and robust to the choice of inner parameters. The presented results are obtained with Jacobi Over-Relaxation method as the linear solver.  }

%{\bf AMS subject classification.} 90C15, 62L20, 60H40 
\end{abstract}

\section{Introduction}

We consider problems of the form 
\be \label{51} \min_{y \in \mathbb{R}^n} f(y)=\sum_{i=1}^{N} f_i(y),
\ee
where $f_i : \mathbb{R}^n \to \mathbb{R}$ are strongly convex  local cost functions.
%These kinds of problems come from  various areas such as machine learning, regression, etc. 
A decentralized optimization framework is  considered, more precisely,  we assume decentralized but connected network of $N$ nodes.

Distributed consensus optimization over networks has become a mainstream research topic, e.g., \cite{wei,BoydADMM,SayedEstimation,scutari,dusan,arxivVersion}, motivated by numerous applications in signal processing \cite{novo1}, control \cite{JoaoMotaMPC}, Big Data analytics \cite{scutari2}, social networks \cite{novo3}, etc. Various methods have been proposed in the literature, e.g., \cite{d1,d2,d3,d4,d5,d6,d7,d8,d9,d10}.

 While early distributed (sub)gradient methods exhibit several useful features, e.g., \cite{nedic_T-AC}, they also have the drawback that they do not converge to the exact problem solution when applied with a constant step-size; that is, for exact convergence, they need to utilize a diminishing step-size \cite{Nedickaskadni}. To address this issue, several different mechanisms have been proposed. Namely, in  \cite{extra}  two different weight-averaging matrices at two consecutive iterations are used.  A gradient-tracking technique where the local updates are modified so that they track the network-wide average gradient of the nodes' local cost functions is proposed and analyzed in \cite{Espectral, harnessing}. The authors of \cite{wei} incorporate multiple consensus steps per each gradient update to obtain the  convergence to the exact solution.

In this paper we investigate a different strategy to develop a novel class of exact distributed methods by employing quadratic penalty approach. The method is defined by the standard reformulation of distributed problem (\ref{51}) into constrained problem in $ \mathbb{R}^{nN} $ with constraints that penalize the differences in local approximations of the solution. The reformulated constrained problem is then solved by a quadratic penalty method. Given that the sequence of penalty subproblems is quadratic, as will be explained further on, we employ a fixed point linear solver to find zeroes of the corresponding gradients. Thus, we abbreviated the method as EFIX - Exact Fixed Point. As it will be detailed further ahead, the EFIX method possesses properties that are at least comparable with existing alternatives in terms of efficiency and  required knowledge of system parameters. 

In more detail,  the proposed approach is as follows. The constrained distributed problem in $ \mathbb{R}^{nN} $ is reformulated by adding a quadratic penalty term that penalizes the differences of solution estimates at neighbouring nodes across the network. Then the sequence of penalty problems are solved inexactly,  wherein the corresponding penalty parameters increase over time to make the algorithm exact. The algorithm parameters, such as the penalty parameter sequence and the levels of inexactness of the  (inner) penalty problems are designed such that the overall algorithm exhibits efficient behaviour. We consider two types of strongly convex objective functions - quadratic and generic strongly convex function. For quadratic objective function the subproblems are clearly quadratic, while in the case of generic function we approximate the objective function at the current iteration with a quadratic model. Hence the penalty subproblems are all quadratic and strongly convex. Solving these problems boils down to finding zeroes of the gradients, i.e. to solving systems of linear equations for each subproblem. To solve these systems of linear equations one can employ any distributed linear solver like fixed point iterative methods.   The proposed framework is general and we exemplify the framework by employing the  Jacobi Over-Relaxation (JOR) method for solving the penalty subproblems.  Numerical tests on both simulated and real data sets demonstrate that the resulting algorithms are (at least) comparable with existing alternatives like \cite{harnessing} in terms of the required computational and communication costs, as well as the required knowledge of global system parameters for proper algorithm execution such as the global (maximal) Lipschitz constant of the local gradients $L$, strong convexity constant $\mu$ and the network parameters.   

From the theoretical point of view the following results are established. First, for the quadratic cost functions, we show that either a sequence generated by the EFIX method is unbounded or it converges to  the exact solution of the original problem \eqref{51}. The worst-case complexity result of order  $\cal{O}(\epsilon^{-1})$ is proved. In the generic case, for strongly convex costs with Lipschitz continuous gradients,  the obtained result corresponds to the well-known result in the classical, centralized optimization - if the iterative sequence converges then its limit is the solution of the original problem. Admittedly, this result is weaker than what is known for existing alternatives like, e.g., \cite{harnessing}, but are enough to theoretically certify the methods and are in line with the general theory of quadratic penalty methods; see, e.g., \cite{Nocedal}. Numerical examples nevertheless demonstrate advantages of the proposed approach. 
Moreover, the convergence results of the proposed method are obtained  although the Linear Independence Constraint Qualification, LICQ is violated. 

It is worth noting that penalty approaches have been studied earlier in the context of distributed consensus optimization, e.g., \cite{qpLinovi, qpLi, scno, Zhou}. The authors of \cite{Zhou} allow for nondifferentiable costs, but their analysis relies on Lagrange multipliers and the distance from a closed, convex feasible set which plays a crucial role in the analysis. In \cite{scno}, a differentiable exact penalty function is employed, but the problem under consideration assumes local constraints and separable objective function. Moreover, LICQ is assumed to hold. In our case, separating the objective function yields the constrained optimization problem \eqref{52} where the LICQ is violated. 
The authors of \cite{qpLi} consider more general problems with possibly nondiffrenetiable part of the objective function and linear constraints and provide the analysis for the decentralized distributed optimization problems in particular (Section 4 of \cite{qpLi}).  They show the convergence to an exact solution by carefully designing the penalty parameters and the step size sequence. The proposed algorithm boils down to the distributed gradient with time-varying step sizes.  The convergence is of the order $\cal{O}(1/\sqrt{k})$, i.e., $\cal{O}(1/k)$ for the accelerated version. Comparing with EFIX, we notice that EFIX algorithm needs the gradient calculations only in the outer iterations, whenever the penalty parameter is increased and a new subproblem is generated, which makes it computationally less demanding. The numerical efficiency of the method in \cite{qpLi} is not documented to the best of out knowledge, although the convergence rate results are very promising. The strong convexity is not imposed in \cite{qpLi}, and possibilities for relaxation of convexity requirements in EFIX are going to be the  subject of further research. The algorithm presented in \cite{qpLinovi} is also based on penalty approach. A sequence of subproblems with increasing penalty parameters is defined and solved by accelerated proximal gradient method. Careful adjustment of algorithmic parameters yields a better complexity result than the results presented here. However, with respect to existing work, the proposed EFIX framework is more general in terms of the subsumed algorithms and can accommodate arbitrary R-linearly-converging solver for quadratic penalty subproblems. Finally, another important advantage of EFIX is the robustness with respect to algorithmic parameters.  

The paper is organized as follows. In Section 2 we give some preliminaries. EFIX method for quadratic problems is defined and analyzed in Section 3. The analysis is extended to general convex case in Section 4 and the numerical results for both quadratic and general case are presented in Section 5. Some conclusions are drawn in Section 6. 

\section{Preliminaries}

The notation we will use further is the following.  With $ \mathbb{A}, \mathbb{B}, \ldots $ we denote  matrices in  $  \mathbb{R}^{nN \times nN}  $ with block elements $ \mathbb{A} = [A_{ij}], \; A_{ij} \in \mathbb{R}^{n \times n} $ and elements $ a_{ij} \in \mathbb{R}. $ Consequently, we denote $ A,B,... \in \mathbb{R}^{n \times n}.  $ The vectors of corresponding dimensions will be denoted as $ \mathbf{x} \in \mathbb{R}^{nN}  $ with components $  x_i \in \mathbb{R}^n $ as well as $ y \in \mathbb{R}^n. $ The norm $ \|\cdot\| $ is the Euclidean norm.

Let us specify more precisely the setup we are considering here.  The network of connected  agents  is  represented by a communication matrix $W = W^T \in \mathbb{R}^{N \times N}$ which is assumed to be  doubly  stochastic. The elements of  $ W $ have the property $w_{i,j}>0$ if and only if there is a direct link between nodes $i$ and $j$. Denote by $O_i$ the set of neighbors of node $i$ and let $\bar{O_i}=O_i \bigcup \{i\}.$ The assumptions on the network are stated as follows.

 \begin{ass} \label{A1}  The matrix $ W  \in \mathbb{R}^{N \times N} $  is  symmetric, doubly stochastic
 and 
 $$ w_{ij} > 0 \mbox{ if } j\in \bar{O}_i, \;  w_{ij} = 0 \mbox{ if }  j \notin \bar{O}_i$$ The network  represented by the communication matrix $W$ is connected and undirected.
 \end{ass}

 Let us assume that each of $N$ nodes has its local cost function $f_i$ and has access to  the corresponding derivatives of this local function. Under the assumption A\ref{A1}, the problem \eqref{51} has the equivalent form 
\be \label{52} \min_{\mathbf{x} \in \mathbb{R}^{nN}} F(\mathbf{x}):=\sum_{i=1}^{N} f_i(x_i) \quad \mbox{s. t. } \quad (\II - \WW)^{1/2} \mathbf{x}=0,
\ee
where $\mathbf{x}=(x_1;...;x_N) \in \RR^{nN}$, $\WW=W \otimes I \in \RR^{nN \times nN}$ and $\II \in \RR^{nN \times nN} $ is the identity matrix.
Therefore,  denoting by $ \LL =\II - \WW$ the Laplacian matrix,  the quadratic penalty reformulation of this problem is
\be \label{53} \min_{\mathbf{x} \in \mathbb{R}^{nN}} \Phi_{\theta} (\mathbf{x}):=F(\mathbf{x}) +\frac{\theta}{2} \mathbf{x}^T \LL \mathbf{x},
\ee
where $\theta >0$ is the penalty parameter. EFIX method proposed in the sequel follows the sequential quadratic programming framework where the sequence of problems \eqref{53} are solved approximately.  

\section{EFIX-Q: Quadratic problems}

Quadratic costs are very important subclass of problems that we consider. One of the typical example is  linear least squares problem which comes from linear regression models, data fitting etc. We start the analysis with the quadratic costs given by 
\be \label{81} f_i(y)=\frac{1}{2} (y-b_i)^T B_{ii} (y-b_i),\ee
where $B_{ii} = B_{ii}^T\in \RR^{n \times n}, b_i \in \RR^n$. Let us denote by $\mathbb{B}=diag(B_{11},...,B_{NN})$ the block-diagonal matrix and $\mathbf{b}=(b_1;...;b_N) \in \RR^{n N}$.  Then, 
$$  \Phi_{\theta}(\mx)= \frac{1}{2} (\mx-\mb)^T \BB (\mx-\mb)+ \frac{\theta}{2} \mx^T\LL \mx $$
and 
 $$ \nabla \Phi_{\theta}(\mx)= (\BB+\theta \LL) \mx -\BB \mb. $$
Thus, solving $\nabla \Phi_{\theta}(\mx)=0$ is equivalent to solving the linear system 
 \be \label{83} \AAA \mx =\mc, \quad \AAA:=\BB+\theta \LL, \quad \mc:=\BB \mb.\ee
Under the following assumptions, this system can be solved in a distributed, decentralized manner by applying a suitable linear solver To make the presentation more clear we concentrate here on the JOR method, without loss of generality. 
\begin{ass} \label{B1}  Each function $f_i, \; i=1,\ldots, N $ is $\mu$-strongly convex. 
 \end{ass}
This assumption implies that the diagonal elements of Hessian matrices $B_{ii}$ are positive, bounded by $\mu$ from below. This  can be easily verified by the fact that  $y^T B_{ii} y\geq \mu \|y\|^2$ for $y=e_j, j=1,...,n$ where $e_j$ is the  $j$-th column of the identity matrix $ I \in \mathbb{R}^{n \times n}. $  Clearly,  the diagonal elements of $\AAA$ are positive. Moreover, $\AAA$ is positive definite with minimal eigenvalue bounded from below  with $\mu$. Therefore,  for arbitrary $ \mx_0 \in \mathbb{R}^{nN} $ and $ \AAA, \mc $ given in (\ref{83}), we can define the JOR iterative procedure as 
\be \label{8} \mx^{k+1}=\MM \mx^k+\pp, \ee
\be \label{9} \MM=q \DD^{-1}\GG+(1-q)\II, \quad \pp=q \DD^{-1} \mc,\ee
where $\DD$ is a diagonal matrix with $d_{ii}=a_{ii}$ for all $i=1,...,nN$, $\GG=\DD-\AAA$,  $\II$ is the identity matrix and $q$ is the relaxation parameter. The structure of $ \AAA $ and $ \MM $ makes the iterative method specified in (\ref{8}) completely distributed assuming that each node $ i $ has the corresponding column of $ \MM,$ and thus we do not need any additional adjustments of the linear solver to the distributed network. The convergence interval for the relaxation parameter $ q $ is well known in this case, see e.g.   \cite{greenbaum}.   
\begin{lemma} \label{L1} Suppose that the assumptions A\ref{A1}-A\ref{B1} are satisfied.  Then the JOR method  converges for all  $q \in (0,2/\rho(\DD^{-1}\AAA))$.
\end{lemma}
The  JOR method (\ref{8})-(\ref{9}) can be stated in the distributed manner as follows. Notice that  the blocks of $ \AAA $ are given by 
\be \label{84} A_{ii}=B_{ii}+\theta (1-w_{ii})I, \quad \mbox{and} \quad  A_{ij}=-\theta w_{ij}I \quad \mbox{for } i \neq j.
\ee
Therefore, we can represent JOR iterative matrix $\MM$  in similar manner, i.e., $\MM=[M_{ij}]$ where 
\be \label{85} M_{ii}=q D_{ii}^{-1}G_{ii}+(1-q)I, \quad 
M_{ij}=q \theta w_{ij} D_{ii}^{-1} \quad \mbox{for} \quad i \neq j, \ee
and  $\pp=(p_1;...;p_N)$ is calculated as
\be \label{86} p_i=q D_{ii}^{-1}B_{ii} b_i. \ee
%where 
%\be \label{87} D_i=diag(H_{ii}), \quad G_i=D_i-H_{ii}. \ee
Thus, each node $i$ can update its own vector $x_i$ by
\be \label{inner} x_i^{k+1}=\sum_{j \in \bar{O}_i}M_{ij} x^k_j+p_i.\ee
Notice that (\ref{inner}) requires only the neighbouring $ x_j^k, $ i.e. the method is fully distributed. The iterative matrix $\MM$ depends on the penalty parameter $\theta,$ so the JOR parameter $q$ needs to be updated for each value of the penalty parameter. Let us now estimate the interval stated in Lemma \ref{L1}. 
We have 
$$\rho(\DD^{-1}\AAA) \leq \|\DD^{-1}\AAA\| \leq \|\DD^{-1}\|\|\AAA\|.$$
Since the diagonal elements of $B_{ii}$ are positive  and $ \DD $ is the diagonal matrix with elements $ d_{ii} = b_{ii} + \theta \ell_{ii}, i=1,\ldots, nN, $ with $ \LL =[\ell_{ij}] \in \mathbb{R}^{nN \times nN},$ we can upper bound the norm of $ \DD^{-1} $ as follows
$$\|\DD^{-1}\|\leq \frac{1}{\theta(1-\bar{w})},$$
where $\bar{w}:=\max_i w_{ii}<1$.  On the other hand, 
$$\|\AAA\|\leq \|\BB\|+2 \theta \leq  \max_i l_i+2 \theta:=L+2 \theta, $$
where $l_i$ is the largest eigenvalue of $B_{ii}$. 
So, the convergence interval for the relaxation parameter can be set as 
\be \label{88} q\in (0,\frac{2\theta(1-\bar{w})}{L+2 \theta}).\ee
Alternatively, one can use the infinity norm and obtain a bound as above with $\bar{B}:=\max_{i} \|B_{ii}\|_{\infty}$ instead of $L$. The iterative matrix depends on the penalty parameter and thus (\ref{88}) can be updated for each penalty subproblem, defined with a new parameter. However the  upper bound in (\ref{88})  is monotonically increasing with respect to $\theta$, so one can set $q \in (0,2 \theta_0 (1-\bar{w})/ (L+2 \theta_0))$ without updating  with the change of $ \theta. $ In the test presented in Section 5  we  use  $\theta_0=2 L$, which further implies that the JOR parameter can be fixed to any positive value smaller than $4(1-\bar{w})/5$. 

 The globally convergent algorithm for problem \eqref{51} with quadratic functions \eqref{81} is given below. In each subproblem we have to solve a linear system of type (\ref{83}). The algorithm is designed such that these linear systems are solved within an inner loop defined by (\ref{inner}).  The penalty parameters $ \{\theta_{s}\} $ with the property $ \theta_{s} \to \infty, \; s \to \infty, $ and the number of inner iterations $ k(s) $ of type (\ref{inner}) are assumed to be given. Also, we assume that the relaxation parameters $ q(s) $ are  defined  by a rule that fulfills (\ref{88}). Thus, for given $ \theta_{s} $ the linear system $ \AAA(\theta_{s}) \mathbf{x} = \mathbf{c} $ is solved approximately in each outer iteration, with the iterative matrix 
 $$ \MM(\theta_{s}) = q(s) \DD^{-1} \GG + (1-q(s))\II. $$ The global constants $ L $ and $ \bar{w}$ are needed for updating the relaxation parameter in each iteration but the nodes can settle them through initial communication at the beginning of iterative process. Thus, they are also treated as input parameters for the algorithm. 

\noindent{\bf Algorithm EFIX-Q}.  

 Given: $\{\theta_{s}\}, $ $ x_i^0 \in \RR^{n}, i=1,...,N$, $\{k(s)\} \subset \mathbb{N},  L, \bar{w}. $  Set $s=0$. 
\begin{itemize}
\item[S1] Set $ k = 0 $ and choose  $q$ according to  \eqref{88} with $\theta=\theta_{s}. $  Let $ \MM = \MM(\theta_{s}), z_i^0 = x_i^{s}, i=1,\ldots,N. $ 
\item[S2] For each $i=1,\ldots, N $ compute the new local solution estimates  
$$z_i^{k+1}=\sum_{j \in \bar{O}_i} M_{ij} z^{k}_j+p_i$$
and set $ k =k+1. $
\item[S3] If $k<k(s)$  go to step S3. Else, set $ \mathbf{x}^{s+1} = (z_1^k,\ldots,z_N^k), \; s=s
+1$ and  go to step S1. 
\end{itemize}

Our analysis relies on the quadratic penalty method, so we state the framework algorithm (see \cite{Nocedal} for example). We assume again that the sequence of penalty parameters $ \{\theta_s\} $ has the property $ \theta_s \to \infty $ and that the tolerance sequence $\{\varepsilon_s\}$  is such that $\varepsilon_s \to 0$.  

\noindent{\bf Algorithm QP}.  

 Given: $\{\theta_{s}\}, \; \{\varepsilon_s\}.$ Set $ s=0. $   
  \begin{itemize}
\item[S1] Find $\mathbf{x}^s$ such that \be \label{68} \|\nabla \Phi_{\theta_{s}}(\mathbf{x}^s)\| \leq \varepsilon_s.\ee 
\item[S2] Set $s=s+1$ and return to S1. 
\end{itemize}

Let us demonstrate that the EFIX-Q fits into the framework of Algorithm QP, that is  given a sequence $\{\varepsilon_s\}$   such that $\varepsilon_s \to 0, $ there exists a proper choice of the sequence  $ \{k(s)\} $ such that  \eqref{68} is satisfied for all penalty subproblems.

\begin{lemma} \label{Lemajmax} Suppose that the assumptions A\ref{A1}-A\ref{B1} are satisfied.  If 
$ \|\nabla \Phi_{\theta_{s}}(\mathbf{x}^s)\| \leq \varepsilon_s $ then $\|\nabla \Phi_{\theta_{s+1}}(\mathbf{x}^{s+1})\| \leq \varepsilon_{s+1}$ for 
\be \label{jmax} k(s)=\Bigg \lceil \Big |\frac{ \log(\mu \varepsilon_{s+1})-\log(L+2 \theta_{s+1})(\varepsilon_s+2\bar{c})}{\log(\rho_{s+1})} \Big |\Bigg \rceil,\ee 
where $ \rho_{s+1} $ is a constant such that $\|\MM(\theta_{s+1})\|\leq \rho_{s+1}<1$ and $\bar{c}=\|\mathbf{c}\|$.
\end{lemma}
\begin{proof} Notice that $\AAA(\theta)$ is positive definite for all $\theta>0$ and thus there exists an unique stationary point $\mathbf{x}^*_{\theta}$ of $\nabla \Phi_{\theta}$, i.e., an unique solution of $\AAA(\theta) \mathbf{x}=\mathbf{c}$.  With notation $ \mathbf{z}^k = (z_1^k; \ldots;z_N^k), \mathbf{z}^0 = \mathbf{x}^s, $ we have 
\begin{eqnarray} \label{100}
\|\nabla \Phi_{\theta_{s+1}}(\mathbf{z}^{k})\|&=& \|\nabla \Phi_{\theta_{s+1}}(\mathbf{z}^{k})-\nabla \Phi_{\theta_{s+1}}(\mathbf{x}^*_{\theta_{s+1}})\| \\ \nonumber
&\leq & \|\AAA(\theta_{s+1})\|\|\mathbf{z}^{k}-\mathbf{x}^*_{\theta_{s+1}}\|\\ \nonumber
&\leq & (L+2 \theta_{s+1})\|\mathbf{z}^{k}-\mathbf{x}^*_{\theta_{s+1}}\| \\ \nonumber
&\leq & (L+2 \theta_{s+1})\rho^{k}_{s+1} \|\mathbf{x}^{s}-\mathbf{x}^*_{\theta_{s+1}}\|
\\ \nonumber
&\leq & (L+2 \theta_{s+1})\rho^{k}_{s+1} (\|\mathbf{x}^{s}-\mathbf{x}^*_{\theta_{s}}\|+\|\mathbf{x}^*_{\theta_{s}}-\mathbf{x}^*_{\theta_{s+1}}\|).
\end{eqnarray}
Let us now estimate the norms  in the final inequality. First, notice  that 
$$\nabla \Phi_{\theta_{s}}(\mathbf{x}^{s})=\nabla \Phi_{\theta_{s}}(\mathbf{x}^{s})-\nabla \Phi_{\theta_{s}}(\mathbf{x}^*_{\theta_{s}})=\AAA(\theta_{s})(\mathbf{x}^{s}-\mathbf{x}^*_{\theta_{s}}).$$
Thus, since $\mu \II \preceq \AAA(\theta_{s})$ we obtain 
\be \label{101} \|\mathbf{x}^{s}-\mathbf{x}^*_{\theta_{s}}\|\leq \|\AAA^{-1}(\theta_{s})\|\|\nabla \Phi_{\theta_{s}}(\mathbf{x}^{s})\|\leq \frac{\varepsilon_s}{\mu}.\ee
Moreover, for any $\theta$ we have 
\be \label{102} \|\mathbf{x}^*_{\theta}\|\leq \|\AAA^{-1}(\theta)\|\|\mathbf{c}\|\leq \frac{\bar{c}}{\mu}.\ee
Putting \eqref{101} and \eqref{102} into \eqref{100} we obtain 
%\be \label{103}
$$ \|\nabla \Phi_{\theta_{s+1}}(\mathbf{z}^{k})\| \leq \frac{(L+2 \theta_{s+1})\rho^{k}_{s+1} (\varepsilon_s+2 \bar{c})}{\mu}. $$
%\ee
Imposing the inequality 
$$\frac{(L+2 \theta_{s+1})\rho^{k}_{s+1} (\varepsilon_s+2 \bar{c})}{\mu}\leq \varepsilon_{s+1},$$
and then applying the logarithm and rearranging,  we obtain that $\|\nabla \Phi_{\theta_{s+1}}(\mathbf{z}^{k})\| \leq \varepsilon_{s+1}$ for all 
$k \geq k(s)$ defined by \eqref{jmax}. Therefore, for $ \mathbf{z}^{k(s)} = \mathbf{x}^{s+1} $  we get the statement. 
\end{proof}

The previous lemma shows that EFIX-Q fits into the framework of quadratic penalty  methods presented above  if we assume $\varepsilon_s \to 0$ and  set $k(s)$ as in \eqref{jmax}, with $ \{\mathbf{x}^s\} $ being the outer iterative sequence of Algorithm EFIX-Q.  Notice that the inner iterations (that rely on JOR method)  stated in steps S2-S3 of EFIX-Q can be replaced with any solver of linear  systems  or any optimizer of quadratic objective function which can be implemented in decentralized manner and exhibits  linear convergence with factor $\rho_s$. Moreover,  it is enough to apply a  solver with R-linear convergence, i.e., any solver  that satisfies
$$\|\mathbf{z}^{k}-\mathbf{x}^*_{\theta_{s+1}}\| \leq  C_{s+1} \|\mathbf{x}^{s}-\mathbf{x}^*_{\theta_{s+1}}\| \rho^{k}_{s+1},$$
where $C_{s+1}$ is a positive constant. In this case,  the slightly modified $k(s)$  with $(L+2 \theta_{s+1})$ multiplied with $C_{s+1}$ in \eqref{jmax} fits the proposed framework. 

Although the LICQ does not hold for (\ref{52}), following the steps of the standard proof and modifying it to cope with LICQ violation, we obtain the global convergence result presented below. 
%We state the full proof for completeness. 

%{\bf{Varijanta A ako je tacno:)}}

\begin{te} \label{exactconvQ} Suppose that the assumptions A\ref{A1}-A\ref{B1} are satisfied. Assume that $\varepsilon_s \to 0$ and  $k(s)$ is defined by  \eqref{jmax}. Let  $\{\mathbf{x}^{s}\} $ be a sequence generated by algorithm EFIX-Q. Then, either the sequence   $\{\mathbf{x}^{s}\} $ is unbounded or it converges to  a solution $\mathbf{x}^*$ of the problem \eqref{52}     and   $x_i^*$ is the solution of problem \eqref{51} for every $i=1,...,N$.
\end{te} 
\begin{proof} Assume that  $\{\mathbf{x}^{s}\} $ is bounded and  consider the problem \eqref{52}, i.e., 
$$\min F(\mathbf{x}), \mbox{ s.t. } h(\mathbf{x}) = 0$$
where $$h(\mathbf{x})=\LL^{1/2} \mathbf{x}. $$
Let $\mathbf{x}^*$ be an arbitrary accumulation point of the bounded sequence $\{\mathbf{x}^{s}\} $ generated by algorithm EFIX-Q, i.e., let 
$$\lim_{s \in K_1} \mathbf{x}^{s}=\mathbf{x}^*.$$
The inequality \eqref{68} implies 
\be \label{69a} \theta_{s} \|\nabla^{T} h(\mathbf{x}^{s}) h(\mathbf{x}^{s})\|-\|\nabla F(\mathbf{x}^s)\| \leq \varepsilon_s.\ee
Since $\nabla^{T} h(\mathbf{x}^s)=(\LL^{1/2})^{T}=\LL^{1/2}$, we obtain 
$$\nabla^{T} h(\mathbf{x}^s) h(\mathbf{x}^s)=\LL \mathbf{x}^s,$$
and  \eqref{69a} implies
\be \label{70a} \|\LL \mathbf{x}^s\| \leq \frac{1}{\theta_{s}} (\|\nabla F(\mathbf{x}^s)\|  + \varepsilon_s).\ee
Taking the limit over $K_1$ we have $\LL \mathbf{x}^*=0$, i.e., $h(\mathbf{x}^*)=0$, so $\mathbf{x}^*$ is a feasible point. Therefore $\WW \mathbf{x}^*=\mathbf{x}^*$, or equivalently $x_1^*=x_2^*=...=x_N^*$, so the consensus is achieved.

Now, we prove that $\mathbf{x}^*$ is an optimal point of problem \eqref{52}. Let us define $\mathbf{\lambda}_s:=\theta_{s} h(\mathbf{x}^s)$. Considering the gradient of the penalty function we obtain 
\be \label{71na} \nabla \Phi_{\theta_{s}}(\mathbf{x}^s)=\nabla F(\mathbf{x}^s)+\theta_{s} \LL \mathbf{x}^{s}=\nabla F(\mathbf{x}^s)+\LL^{1/2}\mathbf{\lambda}_s.\ee
Since $\mathbf{x}^s \to \mathbf{x}^*$ over $K_1$ and $\varepsilon_s \to 0$, from  \eqref{70a} we conclude that $\mathbf{\zeta}_s:=\theta_{s} \LL \mathbf{x}^{s}$ must be bounded over $K_1$.  Therefore, $\mathbf{\lambda}_s=\theta_{s} \LL^{1/2} \mathbf{x}^{s}$ is also bounded over $K_1$  and thus,  there exist $K_2 \subseteq K_1$ and $\mathbf{\lambda}^*$ such that 
\be \label{72a} \lim_{s \in K_2} \mathbf{\lambda}_s = \mathbf{\lambda}^*. \ee 
Indeed, by  the eigenvalue decomposition, we obtain 
$\LL=\UU \VV \UU^T,$ where $\UU$ is an unitary matrix and $\VV$ is the diagonal matrix with eigenvalues of $\LL$. Let us denote them by $v_i$. The matrix is positive semidefinite, so $v_i\geq 0$ for all $i$ and we also know that $\LL^{1/2}=\UU \VV^{1/2} \UU^T$.  Since $\mathbf{\zeta}_s$ is bounded over $K_1$, the same is true for the sequence $\UU^T \mathbf{\zeta}_s=\VV \theta_{s} \UU^T \mathbf{x}^s:=\VV \nu^s.$
Consequently, all the components $v_i [\nu^s]_i$ are bounded over $K_1$ and the same is true for $\sqrt{v_i} [\nu^s]_i$. By unfolding we get that $\VV^{1/2} \theta_{s} \UU^T \mathbf{x}^s$ is bounded over $K_1$ and thus the same holds for 
$$\UU \VV^{1/2} \theta_{s} \UU^T \mathbf{x}^s=\theta_{s} \LL^{1/2} \mathbf{x}^s=\mathbf{\lambda}_s.$$
Now, using \eqref{72a} and  taking the limit over $K_2$ in \eqref{71na}  we get 
%\be \label{73a}
$$  0=\nabla F(\mathbf{x}^*)+\LL^{1/2} \mathbf{\lambda}^*, $$
%\ee
i.e., $\nabla F(\mathbf{x}^*)+\nabla^{T} h(\mathbf{x}^*) \mathbf{\lambda}^*=0, $ which means that $\mathbf{x}^*$ is a KKT point of problem \eqref{52} with $\mathbf{\lambda}^*$ being the corresponding Lagrange multiplier. Since $F$ is assumed to be strongly  convex, $\mathbf{x}^*$  is also a solution of the problem \eqref{52}. Finally, notice that  $x_i^*$ is a solution of the problem \eqref{51} for any given node  $i=1,...,N$.

We have just proved that, for an arbitrary $i$, every accumulation point of the sequence $\{x_i^{s}\} $ is the solution of problem \eqref{51}. Since the function $f$ is strongly convex, the solution of problem \eqref{51} must be unique. So, assuming that there exist accumulation points $\mathbf{x}^*$ and $\tilde{\mathbf{x}}$ such that $\mathbf{x}^* \neq \tilde{\mathbf{x}}$ yields contradiction. Therefore we conclude that all the accumulation points must be the same, i.e., the sequence $\{\mathbf{x}^{s}\} $ converges.  This completes the proof.
\end{proof}

The previous theorem states that the only requirement on $\{\varepsilon_s\}$ is that it is a positive sequence that tends to zero. On the other hand, quadratic penalty function is not  exact penalty function and the solution $\mathbf{x}^*_{\theta}$ of the penalty problem \eqref{53} 
is only an approximation of the solution $y^*$ of problem 
\eqref{51}. Moreover, it is known (see Corollary 9 in \cite{yuan}) that for every $i=1,...,N,$ there holds 
 $$e^1_{i, \theta}:=\|x^*_{i,\theta}-y^*\|=\mathcal{O}(\theta^{-1}).$$ More precisely, denoting by $\lambda_2$ the second largest eigenvalue of $W$ in modulus, we have 
 \be \label{105} e^1_{i, \theta_{s}}\leq \frac{L J}{\theta_{s}\kappa(1-\lambda_2)}\sqrt{4-2 \kappa \theta_{s}^{-1}}+\frac{J}{\theta_{s} (1-\lambda_2)},\ee
 where $\kappa=\mu L/(\mu+L)$ and $J=\sqrt{2 L f(0)}$  since the optimal value of each local cost function is zero. Thus, looking at an arbitrary node $i$ and any outer iteration $s$ we have 
 \be \label{106} \|x_i^s-y^*\|\leq \|x_i^s- x^*_{i,\theta_{s}}\|+ \|x^*_{i,\theta_{s}}-y^*\|:=e^2_{i, \theta_{s}}+e^1_{i, \theta_{s}}.\ee
So, there is no need to solve the   penalty subproblem with more accuracy than $ e^1_{i, \theta}$ -  the accuracy of approximating the  original problem. Therefore, using \eqref{101} and \eqref{105} and balancing these two error bounds we conclude that a suitable value for  $\varepsilon_s,$  see  \eqref{101}, can be estimated as  
\be \label{107} \varepsilon_s=\mu\Bigg(\frac{L J}{\theta_{s}\kappa(1-\lambda_2)}\sqrt{4-2 \kappa \theta_{s}^{-1}}+\frac{J}{\theta_{s} (1-\lambda_2)}\Bigg)\ee
Similar idea of error balance is used in  \cite{Nedickaskadni}, to decide when to decrease the step size. 

Assume that we define $\varepsilon_s$ as in \eqref{107} Together with \eqref{101} we get  
%\be \label{107n1}
$$ \|x_i^s- x^*_{i,\theta_{s}}\|={\cal{O}}\left(\frac{1}{\theta_{s}}\right). $$
%\ee
Furthermore, using  \eqref{105} and \eqref{106} we obtain
%\be \label{107n2}
$$  \|x_i^s-y^*\|={\cal{O}}\left(\frac{1}{\theta_{s}}\right). $$
%\ee 
 Therefore, the following result concerning the outer iterations holds. 
 \begin{prop} \label{complex1} Suppose that the assumptions of Theorem \ref{exactconvQ} hold and that $\varepsilon_s$ is defined by \eqref{107}. Let   $\{\mathbf{x}^{s}\} $ be a bounded sequence generated by  EFIX-Q . Then for  every $i=1,...,N$ there holds 
 %\be \label{prop1}
 $$ \|x_i^s-y^*\|={\cal{O}}\left(\frac{1}{\theta_{s}}\right). $$
 %\ee
 \end{prop}
 The complexity result stated below  for the special  choice of penalty parameters, $ \theta_s = s $ can be easily derived using the above Proposition.

 \begin{cor} \label{complex2} Suppose that the assumptions of Proposition \ref{complex1} hold and $\theta_{s}=s$ for $s = 1,2,\ldots$. 
 %Let   $\{\mathbf{x}^s\} $ be a bounded sequence generated by  EFIX-Q. 
 Then after at most 
 $$\bar{s}=\left\lceil \frac{2J(3+2L/\mu)}{(1-\lambda_2)}\epsilon^{-1}\right\rceil$$
 iterations we have $\|x_i^{\bar{s}}-y^*\|\leq \epsilon$ for all $i=1,...,N$ and any $\epsilon>0$, where $J$ and $\lambda_2$ are as in \eqref{105}.
 %Then for  every $i=1,...,N$ and every outer iteration $s$  there holds 
 %\be \label{cor1} \|x_i^s-y^*\|={\cal{O}}\left(\frac{\log(s)}{t(s)}\right), \quad \mbox{where} \quad t(s)=\sum_{p=0}^{s-1} k(p). \ee
%{\color{blue} Moreover, in order to satisfy $\|x_i^s-y^*\|\leq \epsilon$ for all $i=1,...,N$ and some $\epsilon>0$ we need at most 
% $$\bar{s}=\left\lceil \frac{2J(3+2L/\mu)}{(1-\lambda_2)}\epsilon^{-1}\right\rceil$$
% outer iterations, i.e., at most $t(\bar{s})$ inner iterations, where $J$ and $\lambda_2$ are as in \eqref{105}.}
 \end{cor}
 \begin{proof} 
 %First, notice that the number of inner iterations with fixed $\theta_{s-1}$ satisfies 
% \be \label{co1}
 %$$  k(s-1)={\cal{O}}\left(\log(\theta_{s})\right)={\cal{O}}\left(\log(s)\right). $$
 %\ee 
%Moreover, Proposition \ref{complex1} implies that for any $i \in \{1,\ldots, N\}$ and  any outer iteration $s$ %we have
%\be \label{co2} \|x_i^s-y^*\|={\cal{O}}\left(\frac{1}{s}\right).\ee
%Notice that $t(s)$ given in \eqref{cor1} represents the total number of inner iterations needed to obtain $\mathbf{x}^s$ and thus  
%  $$t(s)=\sum_{p=0}^{s-1} k(p)\leq s \; k(s-1)= s\;  {\cal{O}}\left(\log(s)\right).$$  
%Therefore, we obtain 
%$$\frac{1}{s}\leq {\cal{O}}\left(\frac{ \log(s)}{t(s)}\right)$$
%which, together with \eqref{co2}, {\color{blue} implies the result \eqref{cor1}.

Notice that \eqref{106}, \eqref{105} and \eqref{101} imply  for arbitrary $i$ 
\begin{eqnarray} \nonumber
\|x_i^s-y^*\| &\leq &\frac{\varepsilon_s}{\mu}+e^1_{i, \theta_{s}}\leq 2 \left(\frac{L J}{\theta_{s}\kappa(1-\lambda_2)}\sqrt{4-2 \kappa \theta_{s}^{-1}}+\frac{J}{\theta_{s} (1-\lambda_2)} \right) \\\nonumber
&\leq& \frac{2 J}{\theta_{s}(1-\lambda_2)}\left(\frac{2 (\mu+L)}{\mu}+1\right) \leq \frac{2 J}{\theta_{s}(1-\lambda_2)}(3+2L/\mu).
\end{eqnarray} 
For $\theta_s=s$, the right-hand side of the above inequality is smaller than $\epsilon$  for 
\be \label{complexitys} s\geq \frac{2J(3+2L/\mu)}{(1-\lambda_2)}\epsilon^{-1}\ee which completes the proof.
 \end{proof}
 
Notice that the number of outer iterations $\bar{s}$ to obtain the $\epsilon$-optimal point depends directly on $J$, i.e., on $f(0)$   and the Lipschitz constant $L$. Moreover, it also depends on the network parameters - recall that $\lambda_2$ represents the second largest eigenvalue of the matrix $W$, so the complexity constant can be diminished if we can chose the matrix $W$ such that  $\lambda_2$ is as small as possible for the given network.  
 
\section{EFIX-G: Strongly convex problems}

In this section, we consider strongly convex  local cost functions $f_i$ that are  not necessarily quadratic. The main motivation comes from machine learning problems such as logistic regression where the Hessian is easy to calculate and, under regularization, satisfies Assumption A\ref{B1}. The main idea now is to approximate the objective function with a  quadratic model at each outer iteration $s$ and exploit the previous analysis.  Instead of solving \eqref{68}, we form a quadratic approximation $Q_{s}(\mx)$ of the penalty function $\Phi_{\theta_{s}}(\mx)$ as 
\begin{eqnarray}
\label{201} Q_{s}(\mx)&:=& F(\mx^{s-1})+\nabla^T F(\mx^{s-1})(\mx-\mx^{s-1})+\\ \nonumber &+& \frac{1}{2} (\mx-\mx^{s-1})^T \nabla^2 F(\mx^{s-1})(\mx-\mx^{s-1})+\frac{\theta_{s}}{2} \mx^T \LL \mx
\end{eqnarray} 
and search for $\mx^s$ that satisfies 
\be \label{202} \| \nabla Q_{s} (\mx^s)\| \leq \varepsilon_s.\ee
In other words, we are solving the system of  linear equations 
%\be \label{203} 
$$ \AAA_s \mx =\mc_s,  $$
%\ee
where 
$$  \AAA_s := \nabla^2 F(\mx^{s-1})+\theta_{s}\LL,  $$
%\ee
%\be \label{205}  
$$ \mc_s:= \nabla^2 F(\mx^{s-1}) \mx^{s-1}-\nabla  F(\mx^{s-1}) . $$
%\ee
Under the stated assumptions, $\AAA_s$ is positive definite with eigenvalues bounded  with $\mu$ from below and the diagonal elements of $\AAA_s$ are strictly positive. Therefore, using the same notation and formulas as in the previous section with $\nabla^2 f_i(x_i^{s-1})$ instead of $B_{ii}$ in \eqref{84} we obtain the same bound for the JOR parameter,  \eqref{88}. 

Before stating the algorithm, we repeat the formulas for completeness.  The matrix $\AAA_s=[A_{ij}]$  has blocks $A_{ij} \in \RR^{n \times n}$ given by 
\be \label{84G} A_{ii}=\nabla^2 f_i(x_i^{s-1})+\theta_s (1-w_{ii})I, \quad \mbox{and} \quad  A_{ij}=-\theta_{s} w_{ij}I \quad \mbox{for } i \neq j.
\ee
The JOR iterative matrix is $\MM_s=[M_{ij}]$ where 
\be \label{85G} M_{ii}=q_s D_{ii}^{-1}G_{ii}+(1-q)I, \quad 
M_{ij}=q_s \theta_{s} w_{ij} D_{ii}^{-1} \quad \mbox{for} \quad i \neq j, \ee
and the vector  $\pp_s=(p_1;...;p_N)$ is calculated as $\pp_s=q \DD_s^{-1} \mc_s$, where $\DD_s$ is a diagonal matrix with $d_{ii}=a_{ii}$ for all $i=1,...,nN$ and   $\GG_s=\DD_s-\AAA_s$, i.e., 
\be \label{86G} p_i=q D_{ii}^{-1}c_i, \quad \mbox{where} \quad c_i=\nabla^2 f_i(x_i^{s-1})x_i^{s-1}-\nabla  f_i(x_i^{s-1}).\ee
%where $c_i=\nabla^2 f_i(x_i^{s-1})x_i^{s-1}-\nabla  f_i(x_i^{s-1})$,  .}
%\be \label{87G} \DD =diag(a_{11},\ldots,a_{nN,nN}), \quad \GG=\DD-\AAA_s. \ee}
%\quad a_{s,i}=\nabla^2 f_i(x_i^{s-1})x_i^{s-1}-\nabla  f_i(x_i^{s-1}). \ee

The algorithm presented below is a generalization of EFIX-Q and we assume the same initial setup: the global constants $ L $ and $ \bar{w} $ are known, the sequence of penalty parameters $ \{\theta_s\} $ and the sequence of inner iterations counters $ \{k(s)\} $ are input parameters for the algorithm. 

\noindent{\bf Algorithm EFIX-G}.  

 Input: $\{\theta_{s}\}, $ $ x_i^0 \in \RR^{n}, i=1,...,N$, $\{k(s)\} \subset \mathbb{N},  L, \bar{w}. $  Set $s=0$. 
\begin{itemize}
\item[S1] Each node $i$ sets $q$ according to \eqref{88} with $\theta=\theta_{s}. $   
\item[S2] Each node calculates  $\nabla  f_i(x_i^{s})$ and $\nabla^2  f_i(x_i^{s})$. Define $ \MM = \MM_s $ given by (\ref{85G}),  $ z_i^0 = x_i^{s}, i=1,\ldots,N $ and set $k=0$.
\item[S3] For $ i =1,\ldots, N $ update the  solution estimates 
$$z_i^{k+1}=\sum_{j \in \bar{O}_i} M_{ij} z^{k}_j+p_i$$
and set $ k =k+1. $
\item[S4] If $k<k(s)$  go to step S4. Else, set $ \mathbf{x}^{s+1} = (z_1^k;\ldots;z_N^k), \; s=s
+1$ and  go to step Sthe. 
\end{itemize}

The algorithm differs from the quadratic case EFIX-Q in step S2, where the gradients and the Hessians are calculated in a new point at every outer iteration. Following the same ideas as in the proof of Lemma \ref{Lemajmax}, we obtain the similar result under the following additional assumption. 
\begin{ass} \label{B1G}  For each $y \in \mathbb{R}^n$ there holds  $\|\nabla^2 f_i(y)\| \leq l_i$, $i=1,...,N$.
 \end{ass}
 Notice that this assumption implies that $\|\nabla^2 F(\mx)\| \leq L:=\max_{i} l_i$.

\begin{lemma} \label{LemajmaxG} Suppose that Assumptions A\ref{A1}-A\ref{B1G} hold.  If $  \| \nabla Q_{s} (\mx^s)\| \leq \varepsilon_s$  holds then 
$\|\nabla Q_{s+1}(\mx^{s+1})\| \leq \varepsilon_{s+1}$ for 
\be \label{jmaxG} k(s)=\Bigg \lceil \Big |\frac{ \log(\mu \varepsilon_{s+1})-\log(L+2 \theta_{s+1})(\varepsilon_s+\bar{c}_s+\bar{c}_{s+1})}{\log(\rho_{s+1})} \Big |\Bigg \rceil,\ee 
where $  \rho_{s+1} $ is a constant such that $\|\MM_{s+1}\|\leq \rho_{s+1}<1$ and $\bar{c}_s=\|\mathbf{c}_s\|$.
\end{lemma}

The Lemma above implies that EFIX-G is a penalty method with the penalty function $ Q $ instead of  $\Phi, $   i.e., with \eqref{202} instead of \eqref{68}. 
Notice that due to assumption A\ref{B1}, without loss of generality we can assume that the functions $f_i$ are nonnegative and thus 
 the relation between $\varepsilon_s$ and $\theta_{s}$ can remain as in \eqref{107}. We have the following convergence result which corresponds to the classical statement in centralized optimization, \cite{Nocedal}. %{\bf Ovo bi trebalo da popravimo ako mo\v{z}emo, ja nisam uspela da nadjem dobru referencu sem slajdova C. Cartis, a to bas i nije dobar izvor za citiranje}

\begin{te} \label{exactconvG1} Let the assumptions A\ref{A1}-A\ref{B1G} hold.    Assume that $ \{\mx^s\} $ is a sequence generated by Algorithm EFIX-G such that $k(s)$ is defined by  \eqref{jmaxG} and  $\varepsilon_s \to 0. $ If $ \{\mx^s\} $ is bounded then  every accumulation point  of $ \{\mx^s\} $ is feasible  for the problem \eqref{52}.  Furthermore, if $ \lim_{s \to \infty} \mx^s = \mx^* $ then $ \mx^* $ is the solution of problem  \eqref{52}, i.e., $x_i^*$ is the solution of problem \eqref{51} for every $i=1,...,N$.
\end{te} 
\begin{proof} Let us consider the problem \eqref{52} and denote $h(\mx)=\LL^{1/2} \mx. $
Let $\tilde{\mx}= \lim_{s \in K} \mx^s$ be an arbitrary accumulation point. 
 Notice that 
\begin{eqnarray} \label{207} & &  \|\nabla \Phi_{\theta_{s}}(\mx^s)-\nabla Q_{s}(\mx^s)\| \\ \nonumber 
&=&\|\nabla F(\mx^s)-\nabla F(\mx^{s-1})+\nabla^2 F(\mx^{s-1})(\mx^s-\mx^{s-1})\| \\ \nonumber 
&\leq& 2 L \|\mx^s-\mx^{s-1}\|:=r_s \end{eqnarray}
and thus the error of the quadratic model $r_s$ is also bounded over $K$.
Now,  inequality \eqref{202} together with the previous inequality implies that 
\be \label{70AG} \|\nabla \Phi_{\theta_{s}}(\mx^s)\| \leq \varepsilon_s+r_s,\ee i.e., we obtain  
%\be \label{70G}
$$  \|\LL \mx^s\| \leq \frac{1}{\theta_{s}} (\|\nabla F(\mx^s)\|  + \varepsilon_s+r_s). $$
%\ee
Taking the limit over $K$ in the previous inequality, we conclude that  $\LL \tilde{\mx}=0$, so the feasibility condition is satisfied, i.e., we have  $\tilde{x}_1=\tilde{x}_2=...=\tilde{x}_N$. 

If $ \lim_{s \to \infty} \mx^s = \mx^* $ we have that the error in quadratic model converges to zero from (\ref{207}), i.e. $ \lim_{s \to \infty}  r_s = 0 $  and thus \eqref{70AG} implies that 
$$\lim_{s \in K} \nabla \Phi_{\theta_{s}}(\mx^s)=0.$$
Following the same steps as in the second part of the proof of Theorem \ref{exactconvQ}, we conclude that $\mx^*$ is optimal and the statement follows. 
\end{proof}

\section{Numerical results}
\subsection{Quadratic case} 
We test EFIX-Q method on quadratic functions \eqref{81}  generated as follows, \cite{Espectral}.  Vectors $b_i$ are drawn from the Uniform distribution on $[1,31]$, independently from each other. Matrices $B_{ii}$ are of the form $B_{ii}=P_i S_i P_i$, where $S_i$ are diagonal matrices with Uniform distribution on $[1,101]$ and  $P_i$ are matrices of orthonormal eigenvectors of $\frac{1}{2} (C_i+C_i^T)$ where  $C_i$ have components drawn independently from the standard Normal distribution. 

The network is formed as follows, \cite{Espectral}. We sample $N$ points randomly and uniformly from $[0,1] \times [0,1]$. Two points are directly connected if their distance, measured by the Euclidean  norm,  is smaller than $r=\sqrt{\log(N)/N}$. The graph is connected. Moreover, if nodes $i$ and $j$ are directly connected, we set $w_{i,j}=1/\max \{deg(i),deg(j)\}$, where $deg(i)$ stands for the degree of node $i$ and $ w_{i,i}=1-\sum_{j\neq i}w_{i,j}$. We test on graphs with $N=30$ and $N=100$ nodes. 

The error metrics  is the following
\be \label{errore} e(x^k):=\frac{1}{N} \sum_{i=1}^{N}\frac{\|x_i^k-y^*\|}{\|y^*\|},\ee
where $y^*\neq 0$ is the exact (unique) solution of problem \eqref{51}.

The parameters are set as follows. The Lipschitz constant is calculated as $L=\max_{i} l_i$, where $l_i$ is the largest eigenvalue of $B_{ii}$. The strong convexity constant is calculated as $\mu=\min_{i} \mu_i$, where $\mu_i>0$ is the smallest eigenvalue of $B_{ii}$. 

The proposed method is denoted by  EFIX-Q  $k(s)$ balance to indicate that we use the number of inner iterations given by \eqref{jmax} where $ L, \mu, \bar{c}$ are calculated at the initial phase of the algorithm and   imposing  \eqref{107} to balance two types of errors as discussed in Section 3. The initial value of the penalty parameter is set to $\theta_0=2L$. The choice is motivated by the fact that  the usual step size bound in many gradient-related methods is  $\alpha <1/(2L)$ and  $1/\alpha$ corresponds to the penalty parameter. Hence, we set $\theta \geq 2L$. Further, the penalty parameter is updated by  $\theta_{s+1}=(s+1) \theta_{s}$.
We tested the Jacobi method, i.e., the relaxation parameter is set to  $q=1$. We also tested  JOR method with the parameter $q=2/3$ but the results are quite similar and hence not reported here. The method is designed to  solve the sequence of quadratic problems up to  accuracy  determined $\{ \varepsilon_s \}. $  Clearly, the precision, measured by $ \varepsilon_s $ determines the computational costs. On the other hand it is already discussed that the error in solving a particular quadratic problem should not be decreased too much given that the quadratic penalty is not an exact method and hence each quadratic subproblem is only an approximation of the original constrained problem, depending on the penalty parameter $ \theta_s $. Therefore, we tested several choices of the inner iteration counter and  parameter update, to investigate the error balance and its influence on the convergence. The method abbreviated as    EFIX-Q $k(s)$  is obtained with  $\varepsilon_0=\theta_0= 2L $,  $\varepsilon_s=\varepsilon_0/s$   for $s>0$,  and $ k(s) $ defined by (\ref{jmax}). Furthermore, to demonstrate the effectiveness of   $k(s)$ stated in (\ref{jmax}) we also report the results from the experiments  where the inner iterations are terminated only if \eqref{68} holds, i.e. without a predefined sequence $ k(s).$ We refer to this method as EFIX-Q stopping. Notice that the exit criterion of EFIX-Q is not computable in the distributed framework and the test reports here are performed only to demonstrate the effectiveness of (\ref{jmax}).

The proposed method is compared with the state-of-the-art method \cite{harnessing,angelia} abbreviated as DIGing $1/(mL)$, where $1/(mL)$ represents the step size, i.e., $\alpha=1/(mL)$ for different values of $m \in \{2,3,10,20,50,100\}$. This method is defined as follows
$$x_i^{k+1}=\sum_{j=1}^{N} w_{ij} x_j^k-\alpha u_i^k, \; u_i^{k+1}=\sum_{j=1}^{N} w_{ij} u_j^k+ B_{ii} (x_i^{k+1}-x_i^{k}), u_i^0=\nabla f_i(x_i^0).$$
The cost of this method, In terms of scalar products, per node and  per iteration can be estimated as $3n$ as $\sum_{j} w_{ij} x_j^k$ takes  $n$  scalar products as well as  $\sum_{j} w_{ij} u_j^k$  and  $B_{ii} (x_i^{k+1}-x_i^{k})$. 

In order to compare the costs, we unfold the proposed EFIX-Q method considering all inner iterations consecutively (so $ k $ below is the cumulative counter for all inner iterations consecutively) as follows
$$x_i^{k+1}=q D_{ii}^{-1} G_{ii} x_i^k+(1-q)x_i^k+q \theta D_{ii}^{-1} \sum_{i \neq j} w_{ij} x_j^k+q D_{ii}^{-1} B_{ii} b_i.$$
Since $D_{ii}$ is  diagonal matrix, $q D_{ii}^{-1} G_{ii} x_i^k$ takes $n+1$ scalar products as well as $D_{ii}^{-1} \sum_{i \neq j} w_{ij} x_j^k$. Moreover, $B_{ii} b_i$ is calculated only once, at the initial phase, so $D_{ii}^{-1} B_{ii} b_i$ costs only 1 scalar product. Therefore, the cost of EFIX-Q method can be estimated as $2n+3$ scalar products per node, per iteration. The difference between EFIX-Q and DIGing can be significant  especially for larger value of  $n$, given the ratio $ 3n $ versus $ 2n + 3. $  Moreover, DIGing method requires  at each iteration the exchange  of two vectors, $x_j$ and $u_j$ among all neighbours, while EFIX requires only the exchange of $x_j$, so it is 50$\%$ cheaper than the DIGing method in terms of communication costs. 
%Therefore, we believe that the comparison with respect to iterations is  more than fare. 

We set $\mx^0=0$ for all the tested methods and consider $n=10$ and $n=100$. Figure 1 presents the errors $e(\mx^k)$ throughout iterations $k $ for $ N = 30 $ and $ N = 100. $ The results for different values of $ n $ appear to be very similar and hence we report only the case $ n = 100.  $
\begin{figure}[htbp]
 %\hspace*{-0.5in}\includegraphics[width=0.45\textwidth, angle=-90]{novon10N303.pdf}\includegraphics[width=0.45\textwidth, angle=-90]{novon10N1003.pdf}\\
    \hspace*{-0.5in}\includegraphics[width=0.45\textwidth, angle=-90]{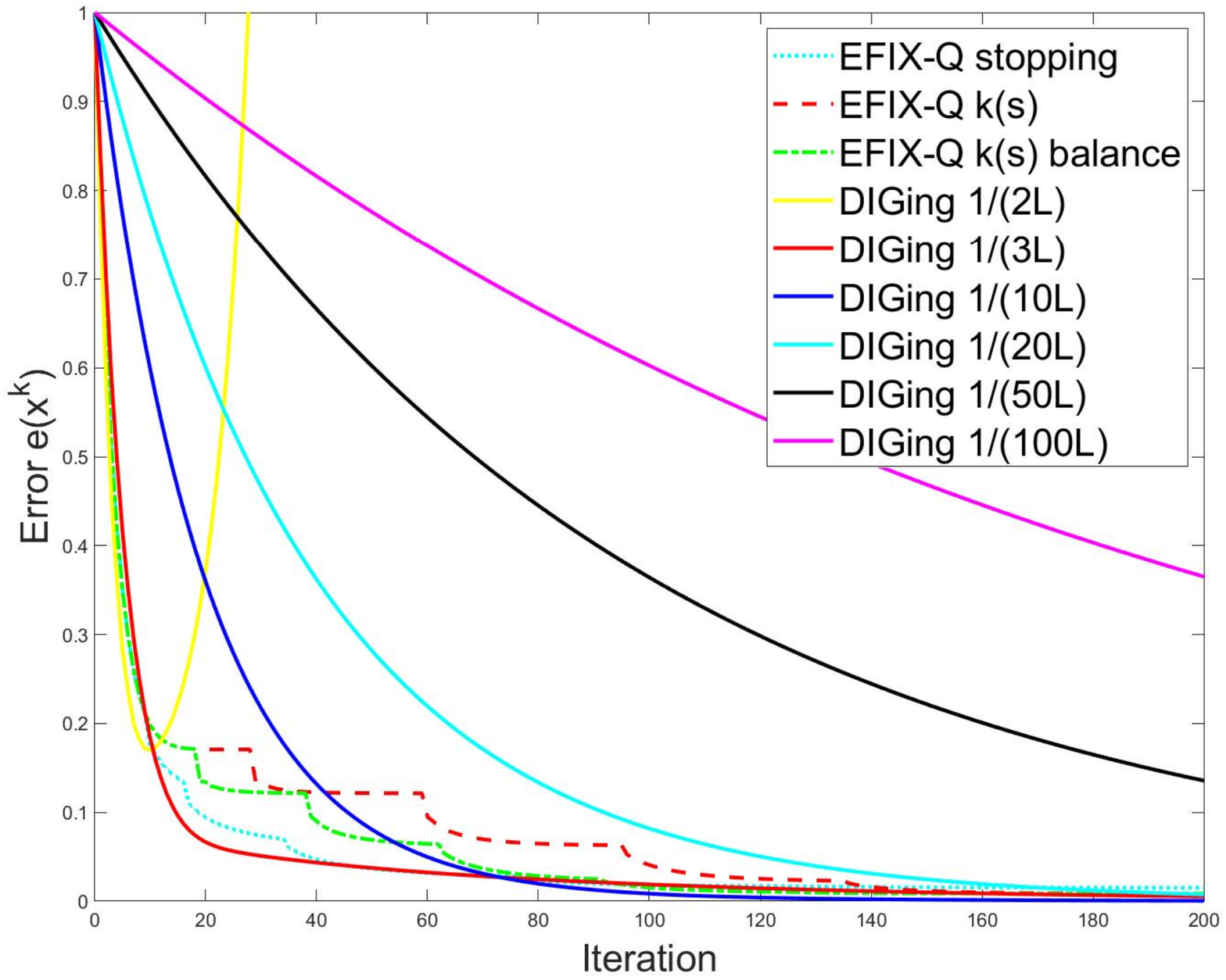}\includegraphics[width=0.45\textwidth, angle=-90]{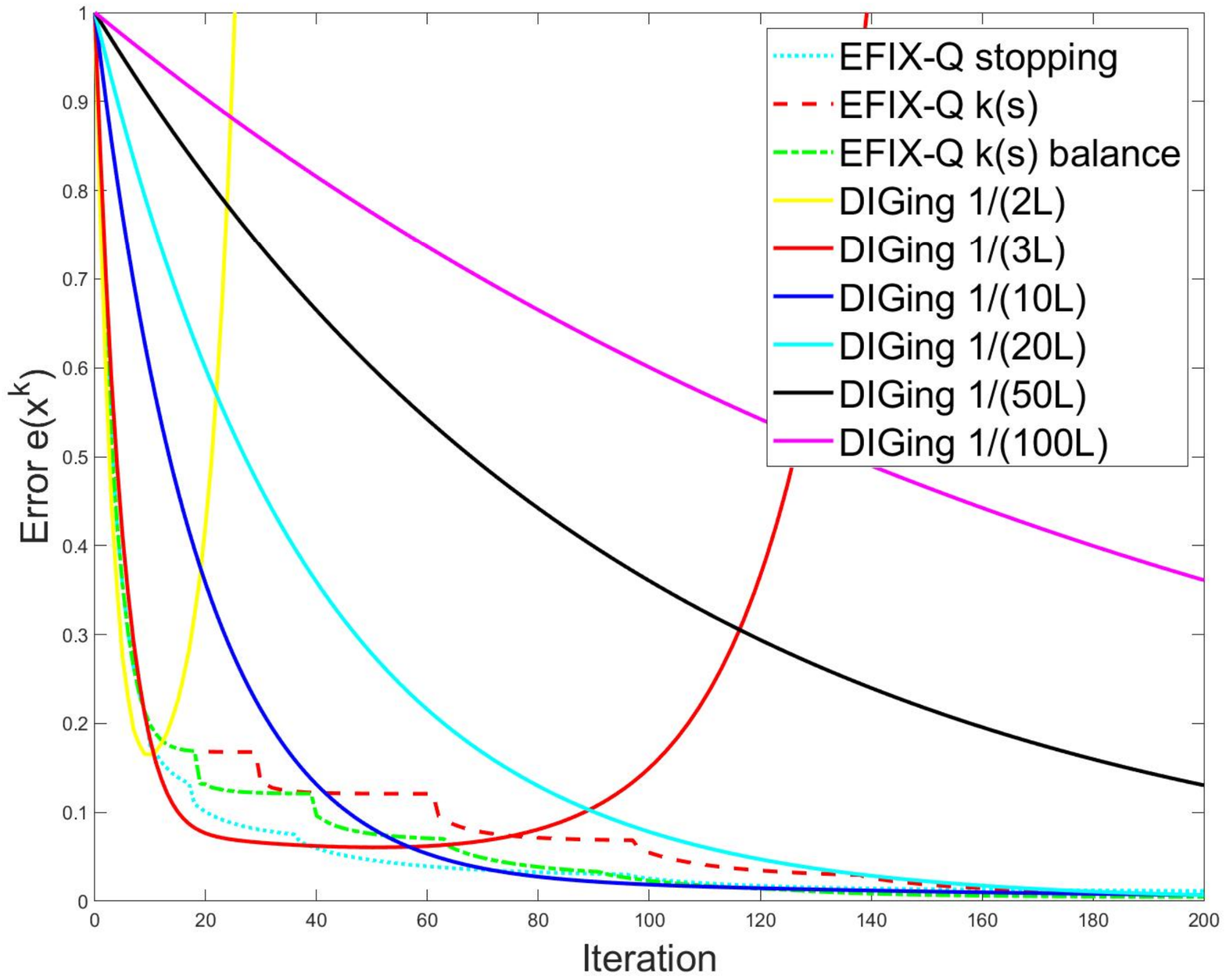}
    \caption{{\footnotesize{The EFIX methods (dotted lines) versus the DIGing method, error \eqref{errore} propagation through iterations for  $ n=100$, $N=30$ (left) and $ n=100$, $N=100$ (right).   }} }
    \label{all1}
\end{figure}

Comparing the number of iterations of all considered methods, from  Figure 1  one can see that  EFIX-Q methods are highly competitive with the best DIGing method in the case of $N=30$. Furthermore,  EFIX-Q outperforms all the convergent DIGing methods in the case of $N=100$.  Moreover, we can see that EFIX-Q $k(s)$ balance behaves similarly to EFIX-Q stopping, so the number of inner iterations $k(s)$ given in Lemma \ref{Lemajmax} is well estimated. Also,  EFIX-Q $k(s)$ balance improves the performance of EFIX-Q $k(s)$  and the balancing of errors yields a more efficient method.  

We compare the tested methods in terms of computational costs, measured by scalar products and communication costs as well. The results are presented in Figure 2 where we compare EFIX-Q $k(s)$ balance with the best convergent DIGing method in the cases  $n=10, N=30$ (top) and $n=100, N=100$ (bottom).  The results show clear advantages of EFIX-Q, especially in the case of larger $n$ and $N$.

\begin{figure}[htbp]
    \hspace*{-0.5in}\includegraphics[width=0.45\textwidth, angle=-90]{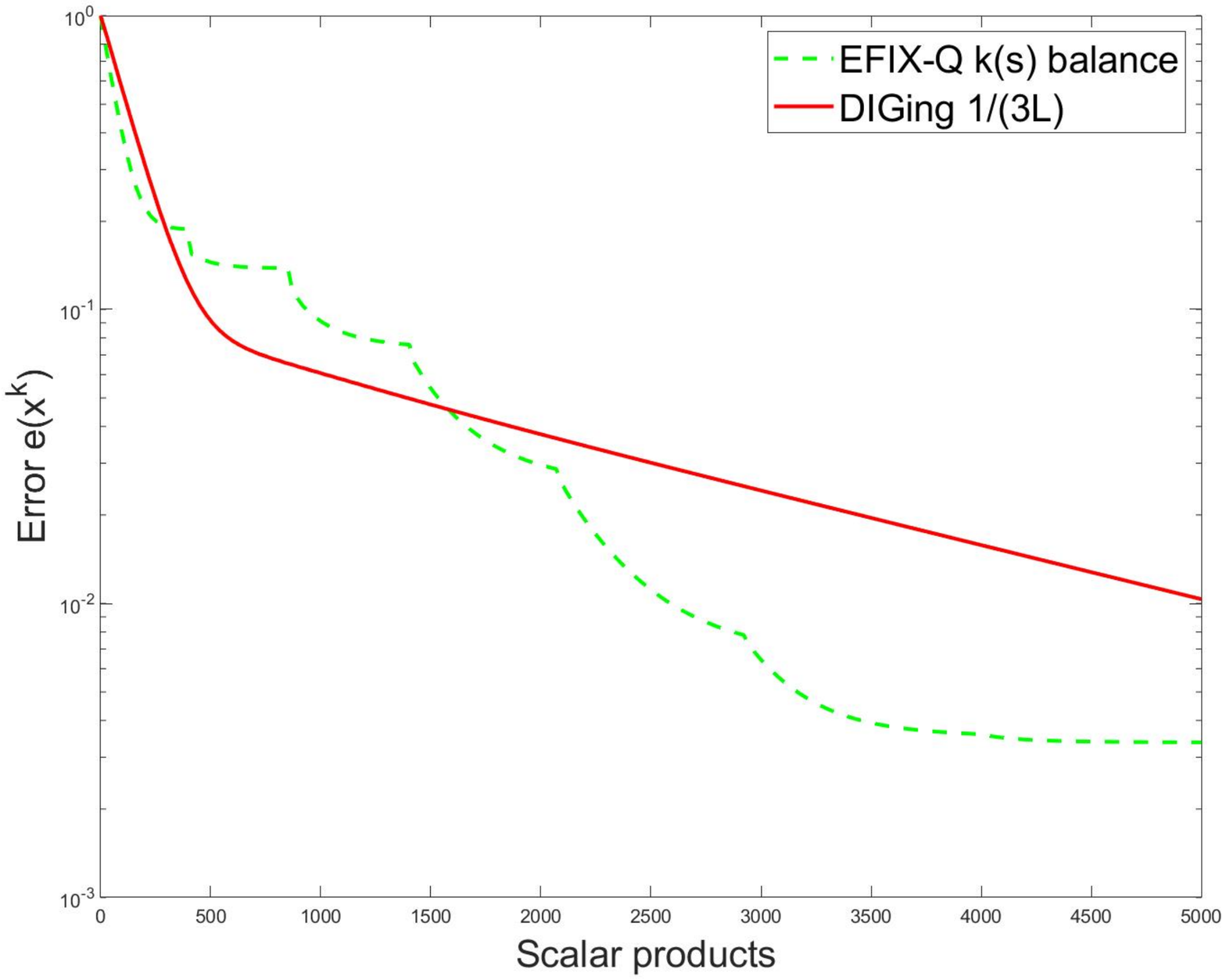}\includegraphics[width=0.45\textwidth, angle=-90]{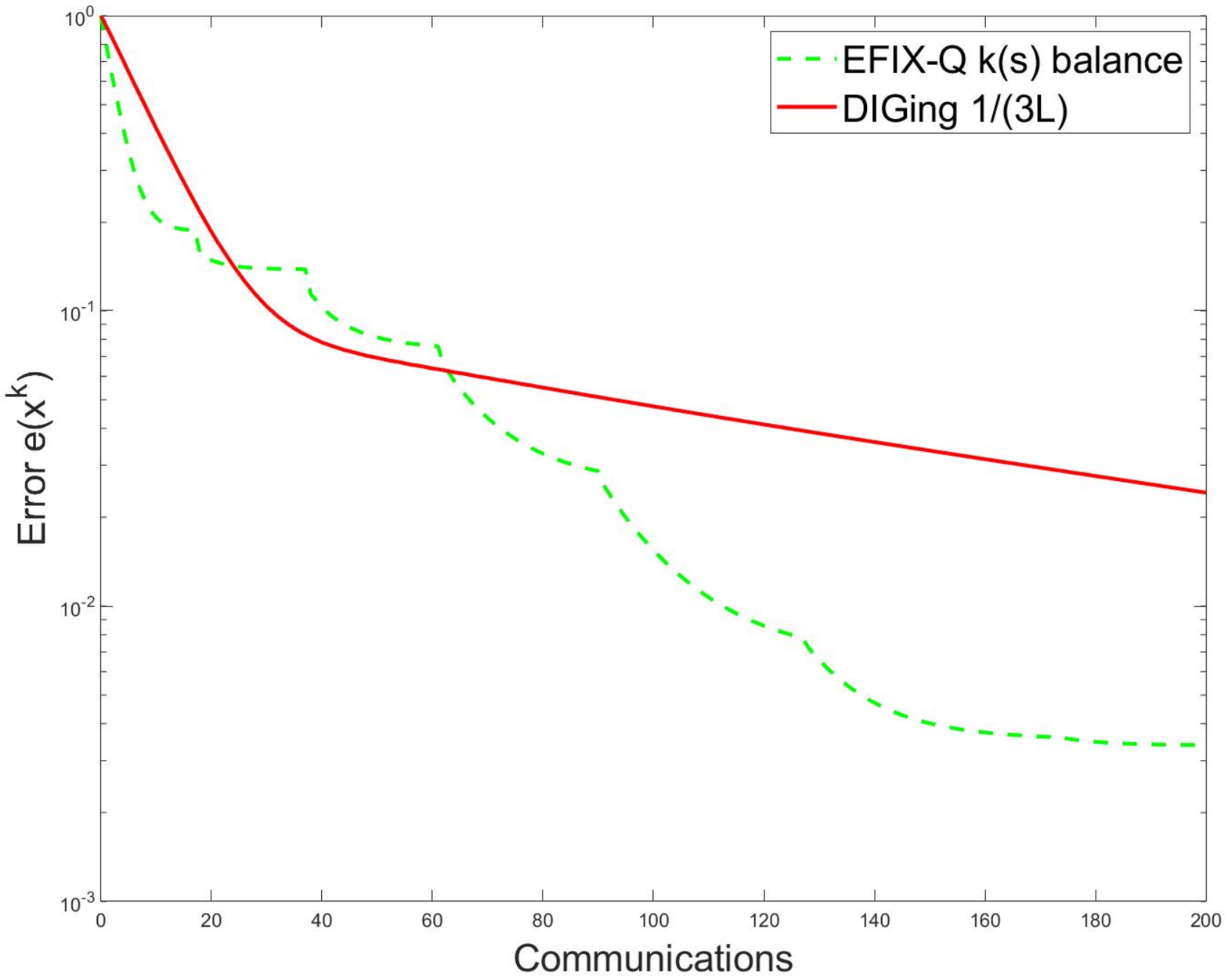}\\
    \hspace*{-0.5in}\includegraphics[width=0.45\textwidth, angle=-90]{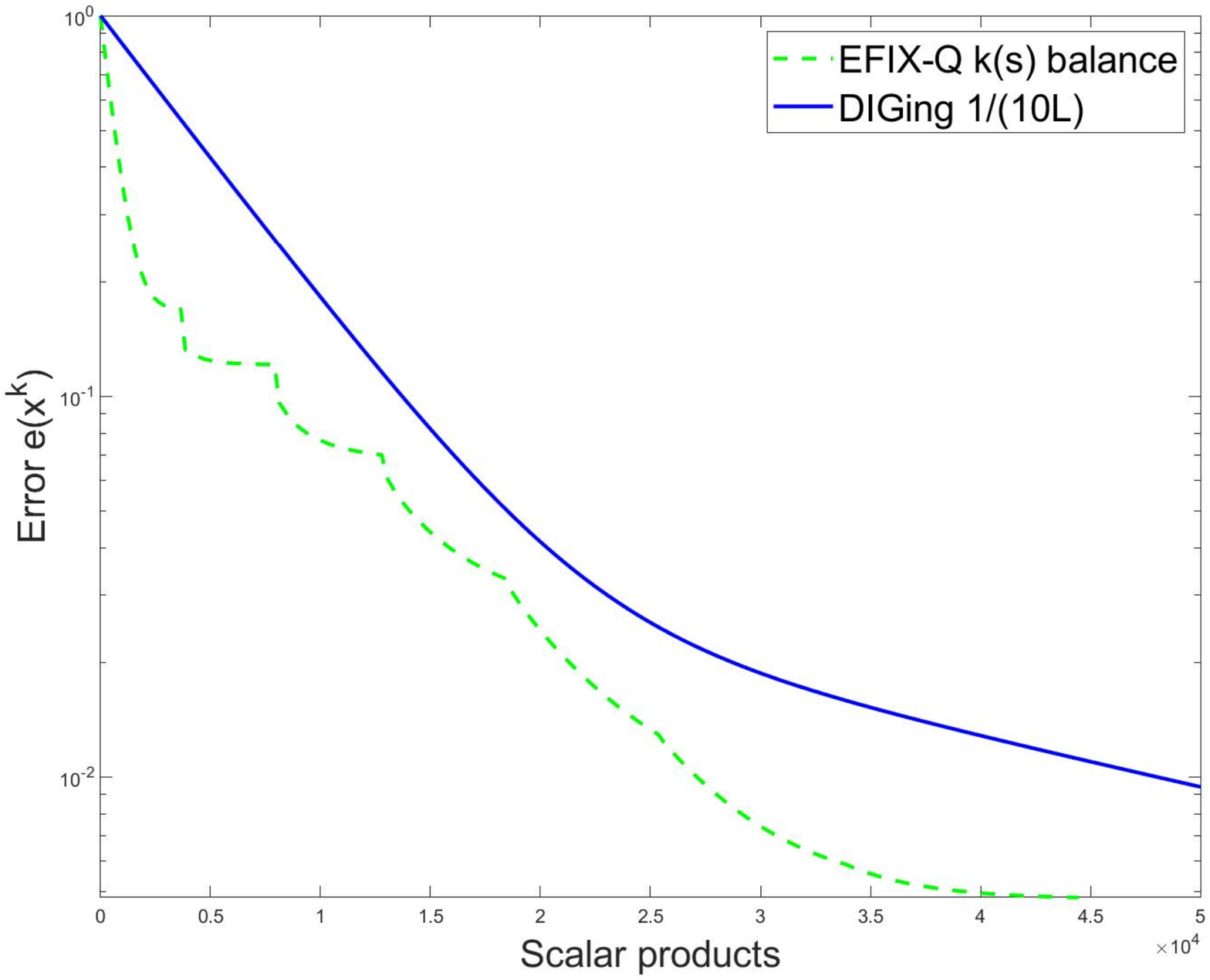}\includegraphics[width=0.45\textwidth, angle=-90]{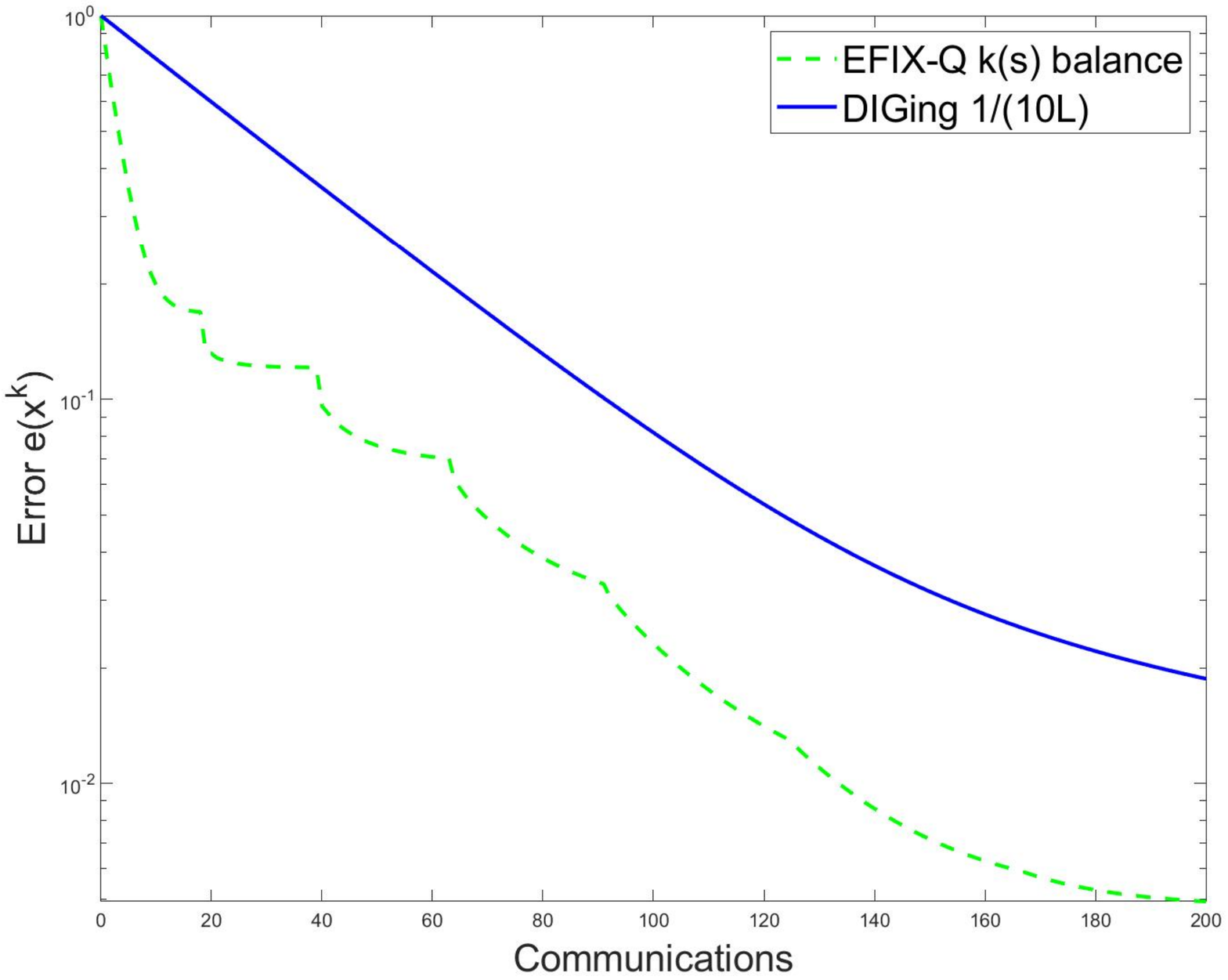}
    \caption{{\footnotesize{The proposed method (dotted line) versus the DIGing method, error \eqref{errore} and the computational cost (left) and communications (right) for $n=10$, $N=30$ (top) and $n=100$, $N=100$ (bottom). }} }
    \label{all1}
\end{figure}

\subsection{Strongly convex problems}
EFIX-G method is tested on the binary classification problems for data sets: Mushrooms \cite{mush} ($n=112$, total sample size $T=8124$), CINA0 \cite{cina0} ($n=132$, total sample size $T=16033$) and Small MNIST \cite{smallmnist} ($n=100$, total sample size $T=7603$). For each of the problems, the data is divided across 30 nodes of the graph described in Subsection 5.1 The logistic regression with the quadratic regularization is used and thus the local objective functions are of the form 
$$f_i(y)=\sum_{j \in J_i} \log(1+e^{-\zeta_j d_j^T y})+\frac{\mu}{2}\|y\|^2:=\sum_{j \in J_i} \tilde{f}_j(y)+\frac{\mu}{2}\|y\|^2,$$
where $J_i$ represents the part of the data assigned to node $i$, $d_j \in \RR^n $ is the corresponding vector of attributes and $\zeta_j \in \{-1,1\}$ represents the label. Evaluation of one local cost function $f_i$ requires $|J_i| +1$ scalar products. However, calculations of  the gradient and the Hessian do not require any additional scalar products since  
$$ \nabla \tilde{f}_j(y)=\frac{1-\psi_j(y)}{\psi_j(y)}\zeta_j d_j+\mu y, \; \nabla^2 \tilde{f}_j(y)=\frac{\psi_j(y)-1}{\psi^2_j(y)}d_j d^T_j+\mu I,$$
$$ \psi_j(y):=1+e^{-\zeta_j d_j^T y}.$$
Moreover, $(\psi_j(y)-1)/\psi^2_j(y) \in (0,1)$ and thus all the local cost functions are $\mu$-strongly convex. The data is scaled in a such way that the Lipschitz constants $l_i$ are 1 and thus $L=1+\mu $. We set $\mu=10^{-4}$.

We test  EFIX-G  $k(s)$ balance, the counterpart of the quadratic version EFIX-Q  $k(s)$ balance, with $k(s)$ defined by \eqref{jmaxG}. The JOR  parameter $q_s$ is set according to \eqref{88}, more precisely, we set  $q=2 \theta_{s} (1-\bar{w})/(L+2 \theta_{s})$. 
A rough estimation of  $\bar{c}_s$ is s $3L\sqrt{N}$ since 
$$ \|\mathbf{c}_s\| \leq  \|\nabla^2 F(\mx^{s-1})\| \|\mx^{s-1}\|+\|\nabla  F(\mx^{s-1})-\nabla  F(\tilde{\mx})\| \leq L 3 \max \{ \|\mx^{s-1}\|, \|\tilde{\mx}\|\}, $$ 
where $\tilde{\mx}$ is a stationary point of the function $F$.
The remaining parameters are set as in the quadratic case. 

Since the solution is unknown in general, the different error metric is used - the average value of the original objective function $f$ across the nodes' estimates  
\be \label{errorv} v(\mx^k)=\frac{1}{N} \sum_{i=1}^{N} f(x_i^k)=\frac{1}{N} \sum_{i=1}^{N} \sum_{j=1}^{N} f_j(x_i^k).\ee

We compare the proposed method with DIGing which takes the following form for general, non-quadratic problems 
$$x_i^{k+1}=\sum_{j=1}^{N} w_{ij} x_j^k-\alpha u_i^k, \; u_i^{k+1}=\sum_{j=1}^{N} w_{ij} u_j^k+ \nabla f_i (x_i^{k+1})-\nabla f_i(x_i^{k}), u_i^0=\nabla f_i(x_i^0).$$
For each of the data sets we compare the methods with respect to iterations, communications and computational costs (scalar products). The communications of the Harnessing method are twice more expensive than for the proposed method, as in the quadratic case. The computational cost of the Harnessing method is $3 n+|J_i|$ scalar products per iteration, per node:  weighted sum of $x_j^k$ ($n$ scalar products); weighted sum of $u_j^k$ ($n$ scalar products); evaluating $\nabla f_i(x_i^{k})$ ($n+|J_i|$ scalar products) because evaluating  of each gradient $\nabla \tilde{f}_j(x_i^k), j \in J_i$ costs 1 scalar product (for $d_j^Tx_i^k$ needed for calculating $\psi_j(x_i^{k})$)  and evaluating the gradient $\nabla f_i(x_i^{k})$ takes the weighted sum of $d_j$ vectors
$$\nabla f_i(x_i^{k})=\sum_{j \in J_i} \frac{1-\psi_j(x_i^{k})}{\psi_j(x_i^{k})}\zeta_j d_j+\mu x_i^{k},$$
 which costs $n$ scalar products. On the other hand, the cost of EFIX-G  $k(s)$ balance per node remains $2n+3$ scalar products at each inner iteration  while in the outer iterations ($s$) we have additional $|J_i|+2n$ scalar products for evaluating $c_i$,  
$$c_i=\sum_{j \in J_i} \frac{\psi_j(x_i^{s-1})-1}{\psi^2_j(x_i^{s-1})}d_j d^T_j x_i^{s-1} + \mu x_i^{s-1} -\nabla  f_i(x_i^{s-1}).$$
Thus we have $|J_i|$ scalar products of the form $d^T_j x_i^{s-1}$, a weighted sum od $d_j$ vectors which costs $n$ SP and the gradient  $\nabla  f_i(x_i^{s-1})$ which costs only $n$ SP  since the scalar products $d^T_j x_i^{s-1}$ are already evaluated and calculated in the first sum. 

The results are presented in Figure 3  $y$-axes is in the log scale). The first column contains graphs for EFIX - G $ k(s) $balance  and all DIGing methods with error metrics through iterations. Obvously, the EFIX -G method is either comparable or better in comparison with DIGing methods. To emphasize the difference in computational  costs we plot in column two the graphs of  error metrics with respect to scalar products for EFIX -G  and the two best DIGing method. The same is done in column three of the graph for the communication costs.

\begin{figure}[htbp]
    \hspace*{-0.5in}%\includegraphics[width=0.4\textwidth]{Mushiter.pdf}
    \includegraphics[width=0.3\textwidth, angle=-90]{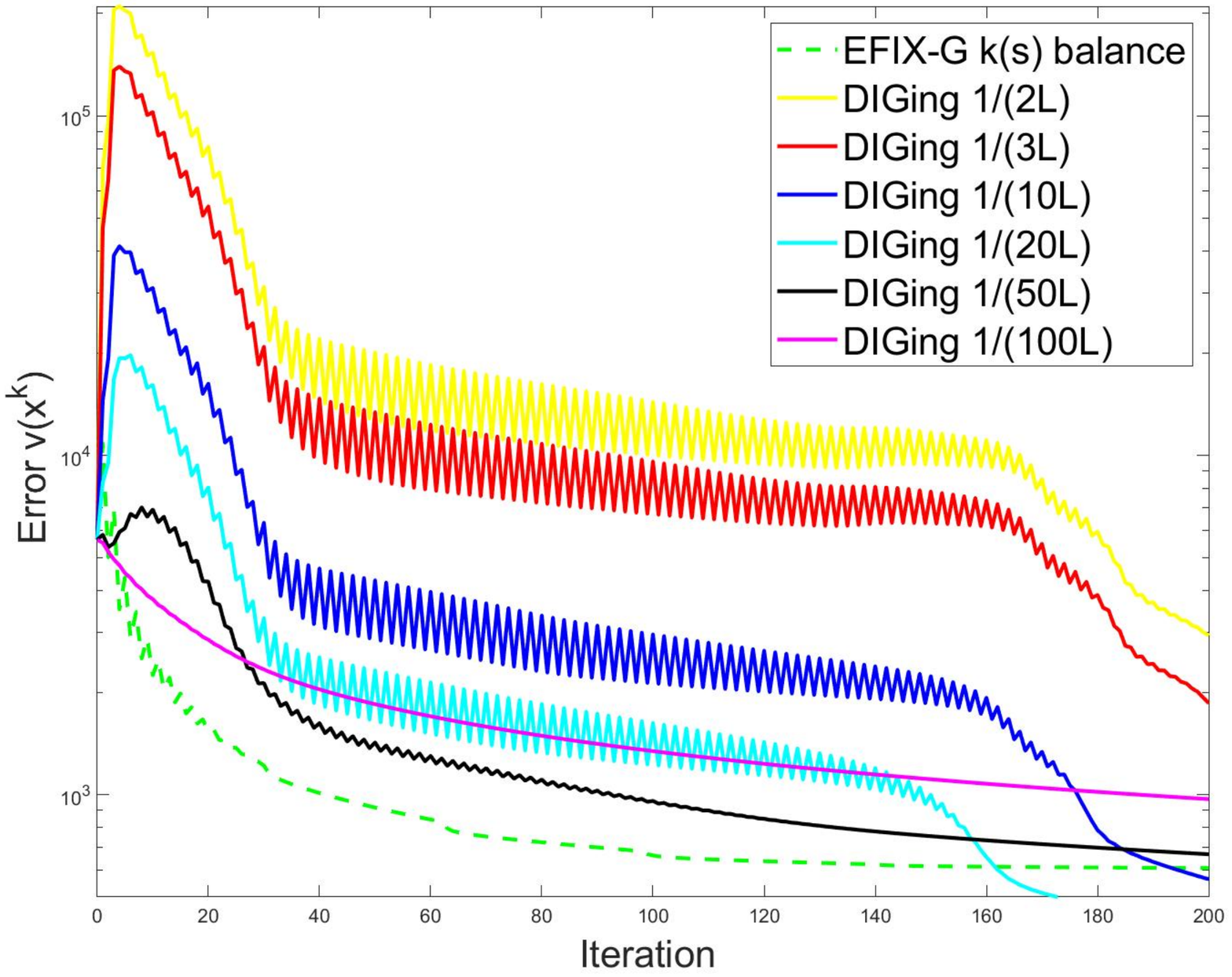}
    \includegraphics[width=0.3\textwidth, angle=-90]{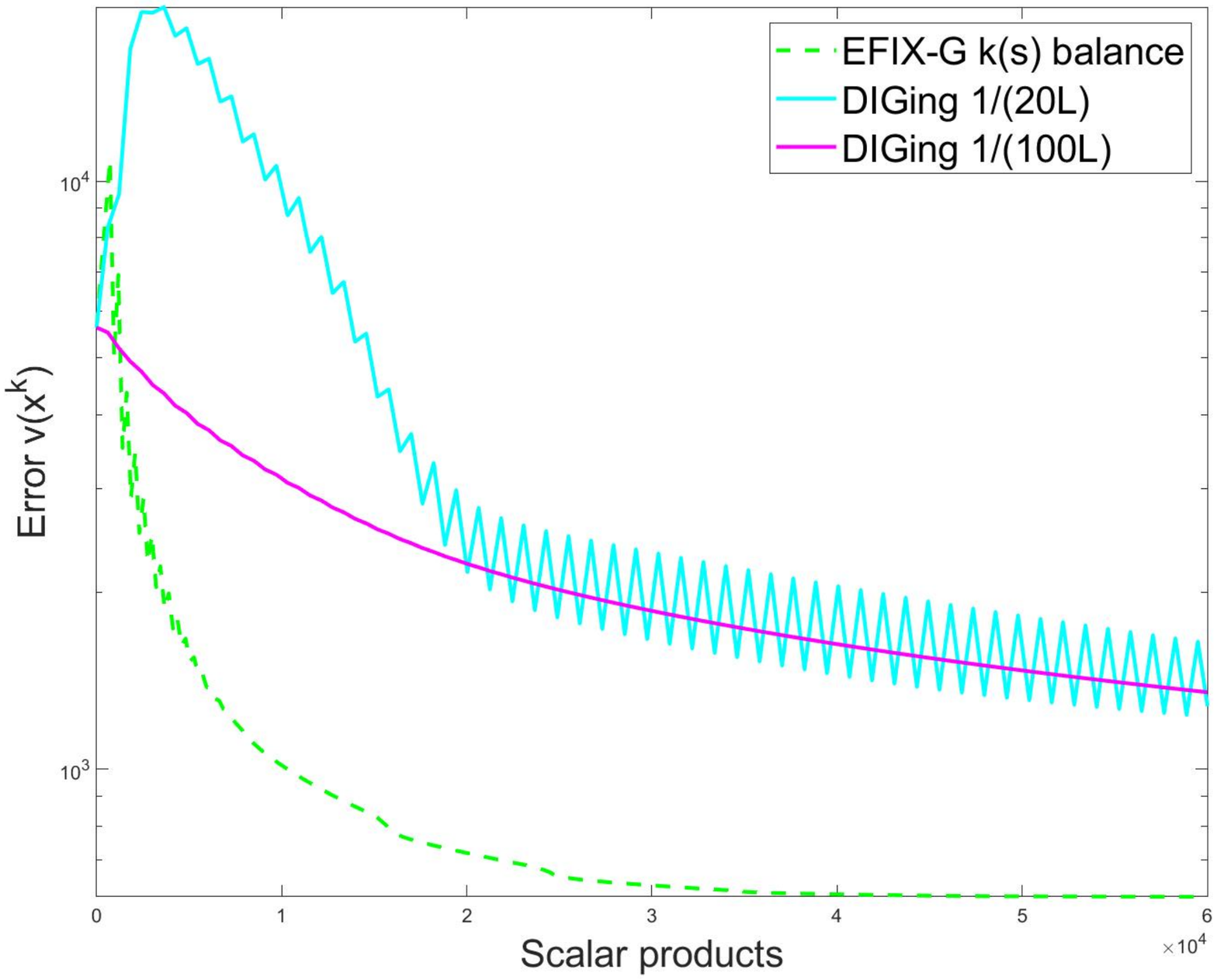}\includegraphics[width=0.3\textwidth, angle=-90]{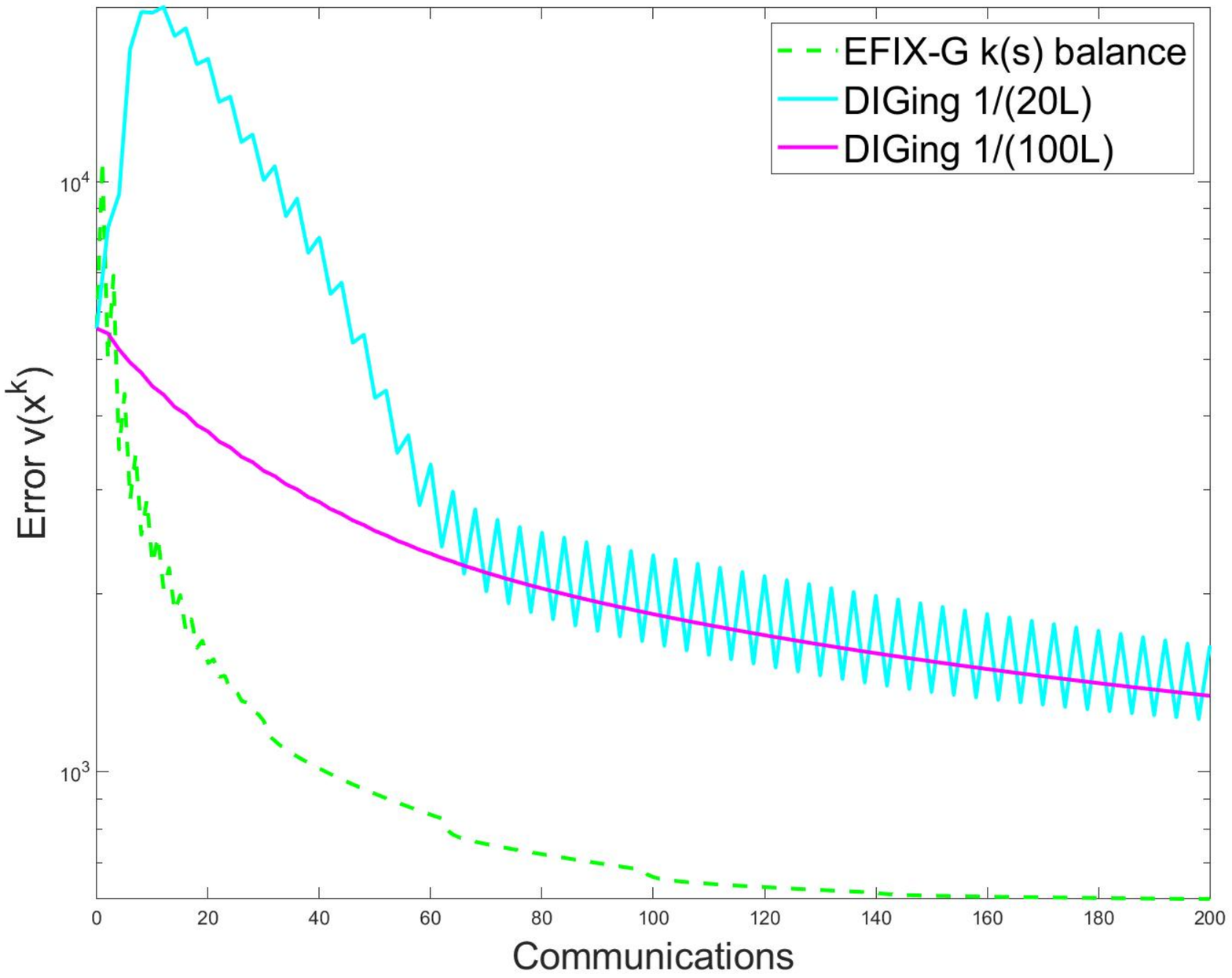}\\
    \hspace*{-0.5in}\includegraphics[width=0.3\textwidth, angle=-90]{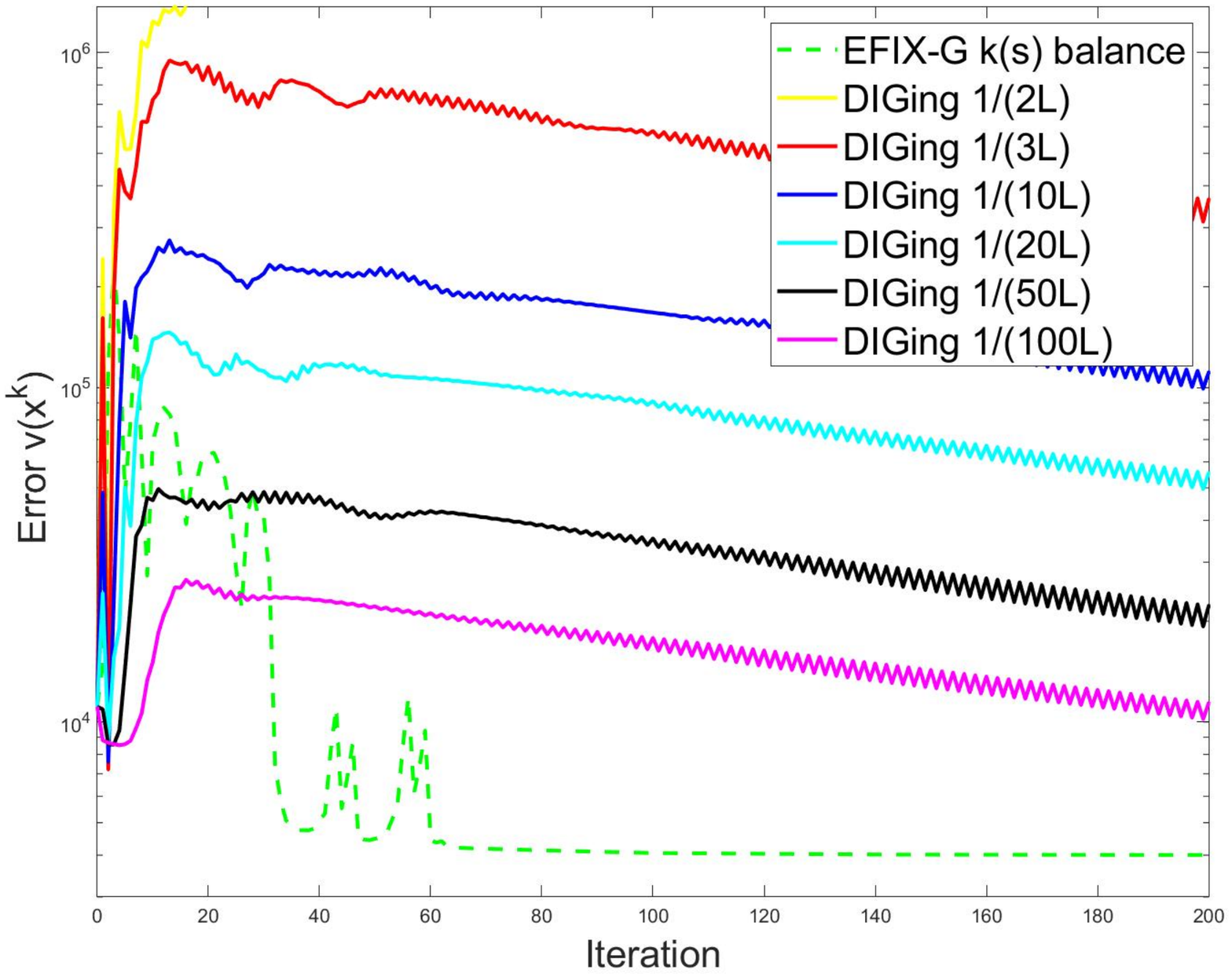}\includegraphics[width=0.3\textwidth, angle=-90]{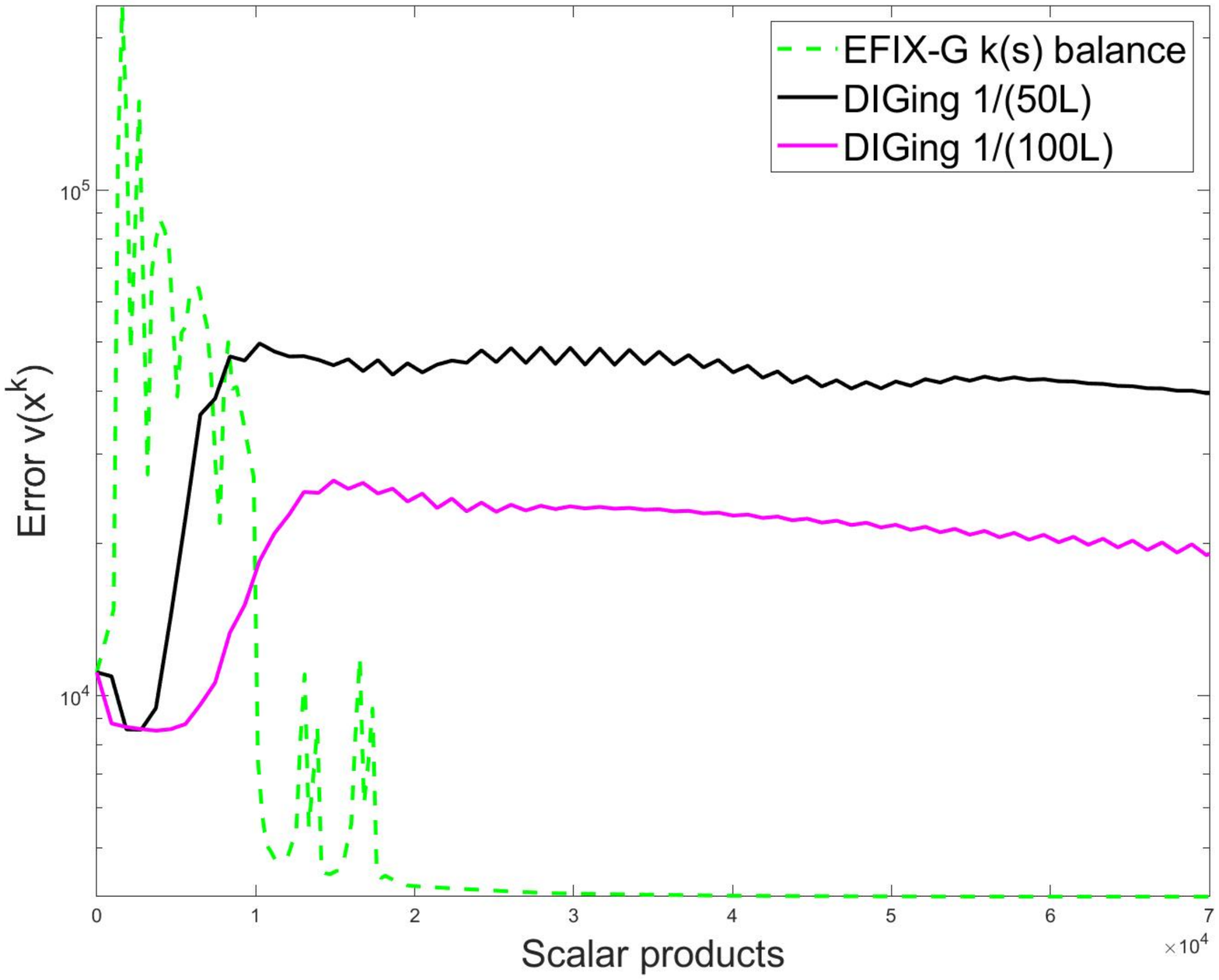}\includegraphics[width=0.3\textwidth, angle=-90]{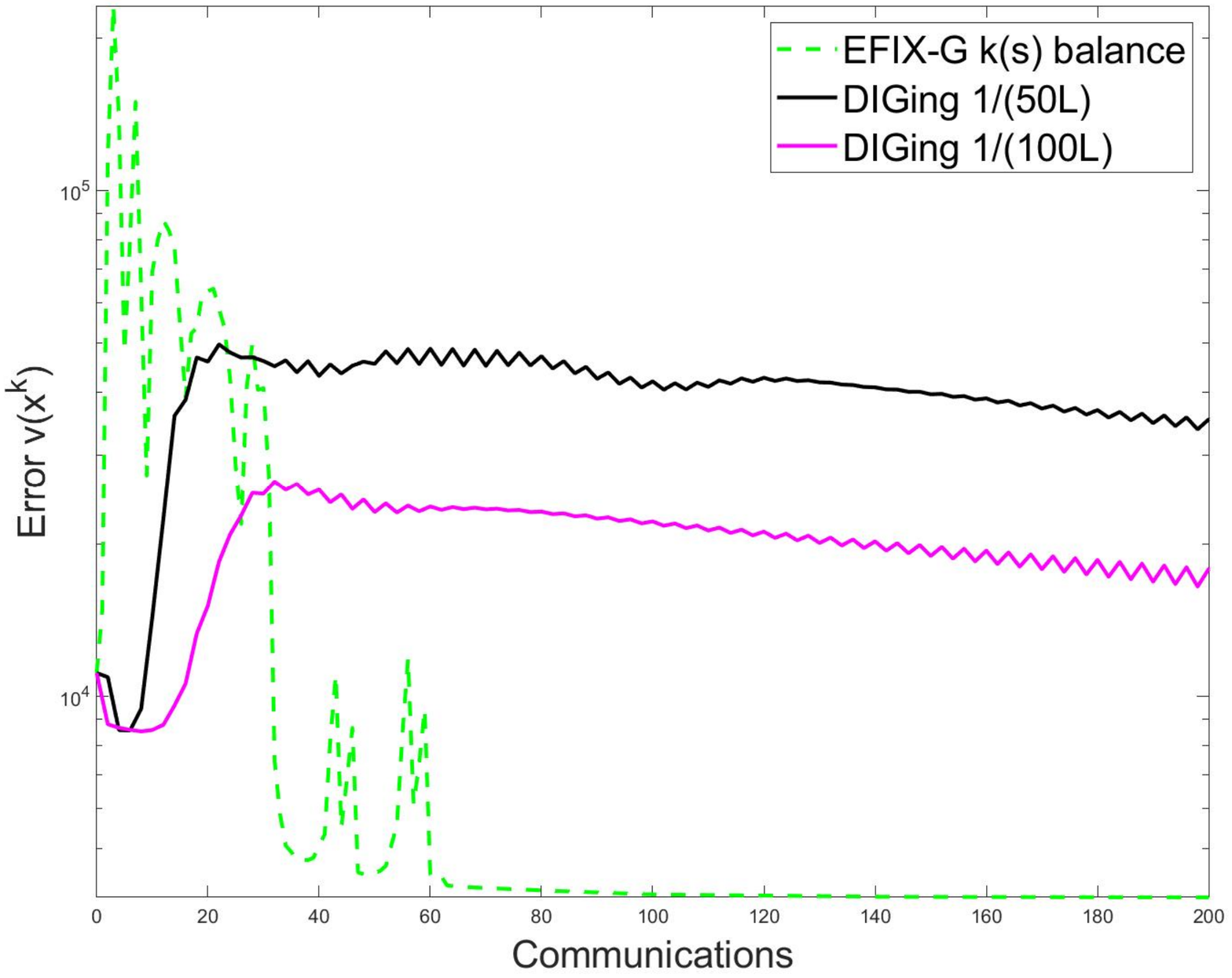}\\
    \hspace*{-0.5in}\includegraphics[width=0.3\textwidth, angle=-90]{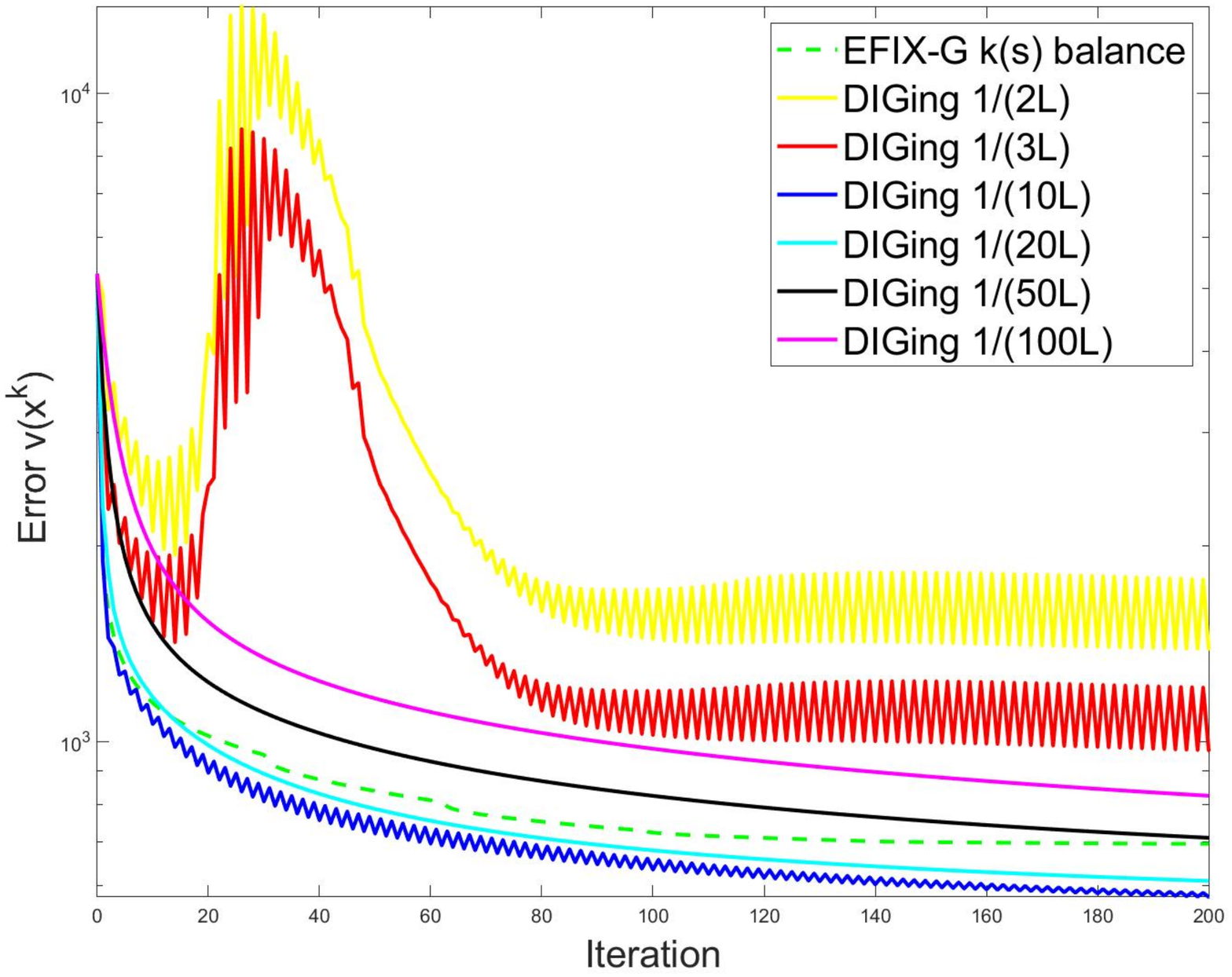}\includegraphics[width=0.3\textwidth, angle=-90]{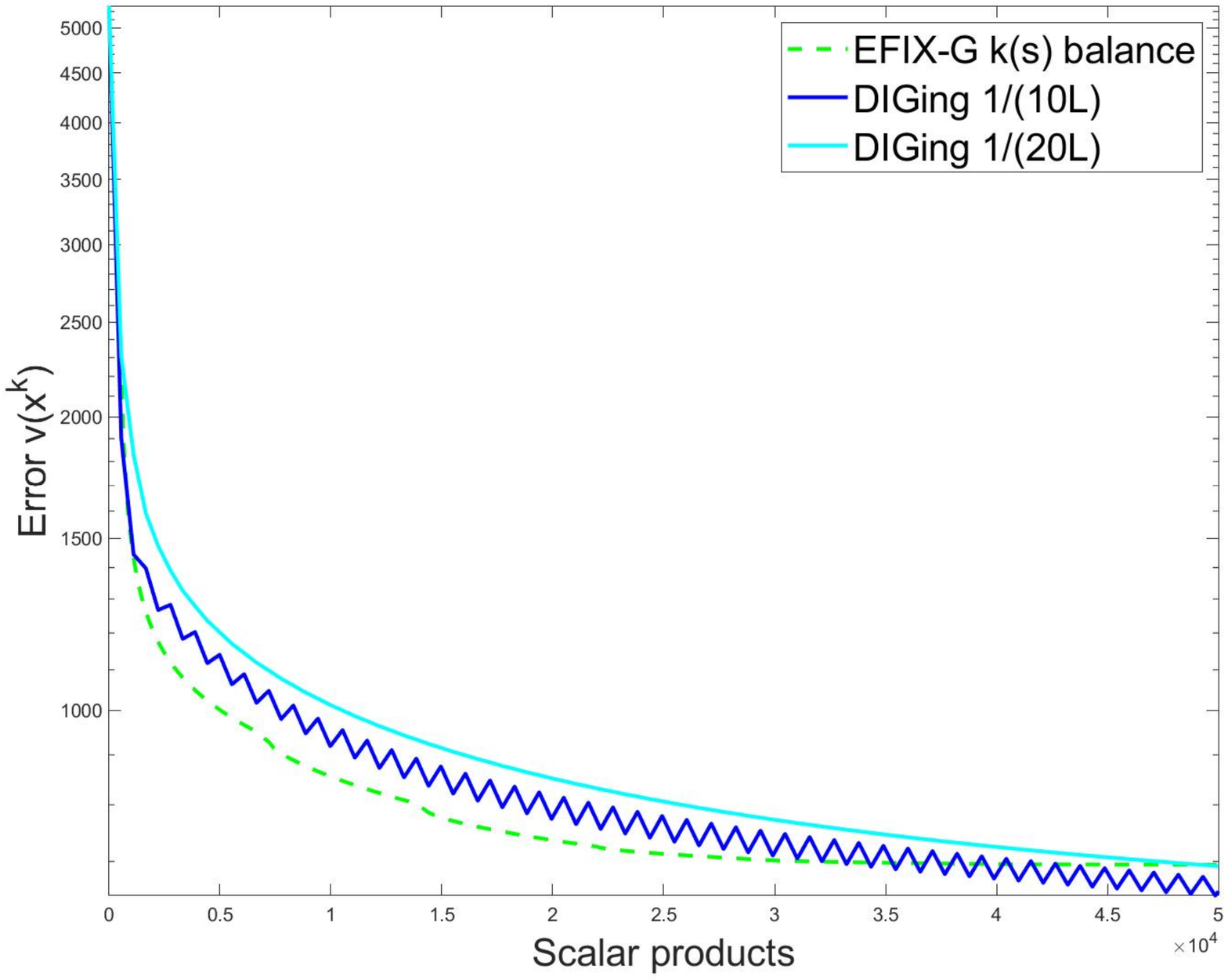}\includegraphics[width=0.3\textwidth, angle=-90]{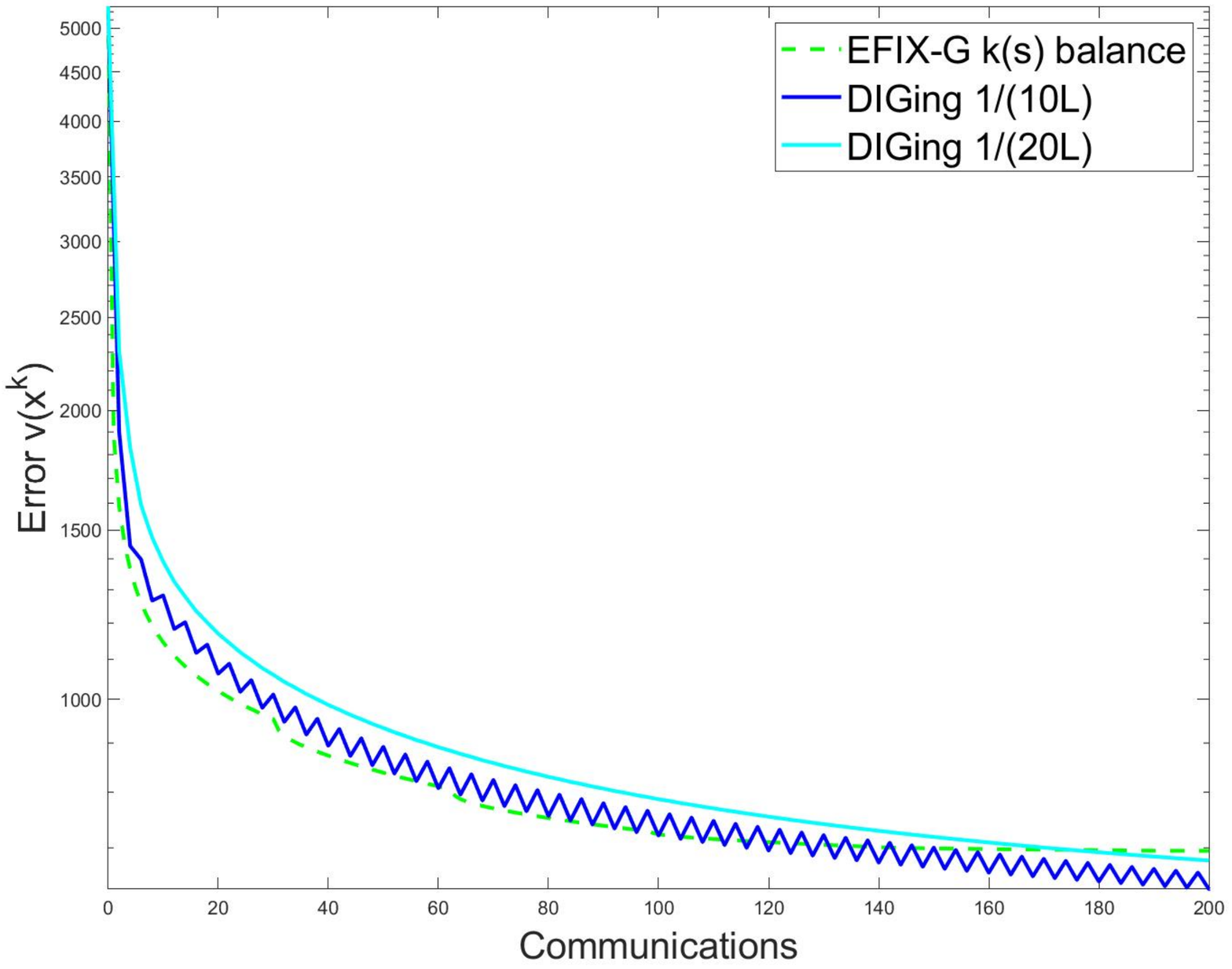}
    \caption{{\footnotesize{The proposed method (dotted line) versus the DIGing method on  Mushrooms (top), CINA0 (middle) and Small MNIST data set (bottom).  }} }
    \label{all2}
\end{figure}

\section{Conclusions}
The quadratic penalty framework is extended to distributed optimization problems. Instead of standard reformulation with quadratic penalty for distributed problems, we define a sequence of quadratic penalty subproblems with increasing penalty parameters. Each subproblem is then approximately solved by a distributed fixed point linear solver. In the paper we used the Jacobi and Jacobi Over-Relaxation method as the linear solvers, to facilitate the explanations. The first class of optimization problems we consider are quadratic problems with positive definite Hessian matrices. For these problems we define the EFIX-Q method, discuss the convergence properties and derive a set of conditions on penalty parameters, linear solver precision and inner iteration number that yield an iterative sequence which converges to the solution of the original, distributed and unconstrained problem. Furthermore, the complexity  bound of $ {\cal O}(\epsilon^{-1}) $ is derived. In the case of strongly convex generic function we define EFIX-G method. It follows the reasoning for the quadratic problems and in each outer iteration we define a quadratic model of the objective function and couple that model with the quadratic penalty. Hence, we are again solving a sequence of quadratic subproblems. The convergence statement is weaker in  this case but nevertheless corresponds to the classical statement in the centralized penalty methods - we prove that if the sequence converges then its limit is a solution of the original problem. 
The method is dependent on penalty parameters, precision of the linear solver for each subproblem and  consequently, the number of inner iterations for subproblems. As quadratic penalty function is not exact, the approximation error is always present and hence we investigated the mutual dependence of different errors. A suitable choice for the penalty parameters, subproblem accuracy and inner iteration number is proposed for quadratic problems and extended to the generic case. 
The method is tested and compared with the state-of-the-art first order exact method for distributed optimization, DIGing. It is shown that EFIX is highly comparable with DIGing in terms of error propagation with respect to iterations and that EFIX computational and communication costs are lower in comparison with DIGing methods.   

\section*{Acknowledgements}
This work is supported by the  Ministry of Education,
Science and Technological Development, Republic of Serbia.

\end{document}